\documentclass[a4paper,10pt]{amsart}

\usepackage[T1]{fontenc}

\usepackage{microtype}

\usepackage[abbrev]{amsrefs}

\numberwithin{equation}{section}

\usepackage{enumerate}

\title[Stationary solutions around local maxima]{Stationary solutions of the nonlinear Schr\"{o}dinger equation 
with fast-decay potentials concentrating around local maxima}

\author{Jonathan Di Cosmo}
\address{
Universit{\'e} catholique de Louvain\\
Institut de Recherche en Math\'ematique et en Physique\\
Chemin du Cyclotron 2 bte L7.01.01\\
1348 Louvain-la-Neuve\\
Belgium}
\address{
  D{\'e}partement de Math{\'e}matique\\
  Universit{\'e} libre de Bruxelles, CP 214\\
  Boulevard du Triomphe, 1050 Bruxelles, Belgium}
\email{Jonathan.DiCosmo@uclouvain.be}
\thanks{Jonathan Di Cosmo is a research fellow of the Fonds de la Recherche Scientifique--FNRS}

\author{Jean Van Schaftingen}
\address{
Universit{\'e} catholique de Louvain\\
Institut de Recherche en Math\'ematique et en Physique\\
Chemin du Cyclotron 2 bte L7.01.01\\
1348 Louvain-la-Neuve\\
Belgium}
\email{Jean.VanSchaftingen@uclouvain.be}

\date{\today}

\newcommand{\R}{\mathbb{R}}
\newcommand{\N}{\mathbb{N}}
\newcommand{\norm}[1]{\left\| #1 \right\|}
\newcommand{\abs}[1]{\left| #1 \right|}
\newcommand{\Bigabs}[1]{\Bigl\lvert #1 \Bigr\rvert}
\newcommand{\dualprod}[2]{\langle #1, #2 \rangle}
\newcommand{\st}{\; \mid\; }

\DeclareMathOperator{\supp}{supp}
\DeclareMathOperator{\dist}{dist} 

\newtheorem{theorem}{Theorem}
\newtheorem{proposition}{Proposition}[section]
\newtheorem{lemma}[proposition]{Lemma}

\keywords{Stationary nonlinear Schr\"odinger equation; semiclassical states; semilinear elliptic problem; vanishing potential; critical frequency; concentration around local maxima}
\subjclass[2010]{35J65 (35B05, 35B25, 35B40, 35J20, 35Q55)}

\begin{document} 

\begin{abstract}
 We study positive bound states for the equation
\begin{equation*}
- \varepsilon^2 \Delta u + Vu = u^p, \qquad  \text{in \(\R^N\)},
\end{equation*}
where \(\varepsilon > 0\) is a real parameter, \(\frac{N}{N-2} < p < \frac{N+2}{N-2}\) and \(V\) is a nonnegative
potential. Using purely variational techniques, we find solutions which concentrate at local maxima of the potential \(V\) without any restriction on the potential.
\end{abstract}

\maketitle

% \tableofcontents

\section{Introduction}
We consider the stationary nonlinear Schr\"{o}dinger equation
\begin{equation}
   \label{problemNLSE}
\tag{\protect{$\mathcal{P}_\varepsilon$}}
\left\{
\begin{aligned}
 - \varepsilon^2 \Delta u + Vu &= u^p, & & \text{in \(\R^N\)},\\
  u(x) & \to 0 &&\text{as \(\abs{x} \to \infty\)},
\end{aligned}
\right.
\end{equation}
where \(\varepsilon > 0\) is a real parameter, \(N\geq 3\), \(\frac{N}{N-2} < p <
\frac{N+2}{N-2}\) and 
\(V \in C(\R^N,\R^+)\).

In the semi-classical limit when \(\varepsilon\) is small, one expects quantum physics to be approximated by classical physics and thus the stationary solutions should concentrate around critical points of the potential. 
A first way to construct such a family of solutions around a \emph{nondegenerate} critical point of the potential is the Lyapunov--Schmidt reduction 
\cites{AmbrosettiBadialeCingolani1996,AmbrosettiBadialeCingolani1997,AmbrosettiMalchiodi2006,FloerWeinstein1986,Oh1988,Oh1988Errata,Oh1990}. 
Solutions of \eqref{problemNLSE} can also be found by variational methods. The most natural method yields solutions concentrating around a \emph{global minimum} of the potential \(V\) \citelist{\cite{Rabinowitz1992}\cite{Wang1993}}. More elaborate critical constructions allow to construct solutions concentrating around \emph{strict local minima} \cites{delPinoFelmer1996,delPinoFelmer1998} and around \emph{strict local maxima} \cites{delPinoFelmer1997,delPinoFelmer2002}.

All the works mentioned above are concerned with \emph{subcritical frequency} case \(\inf_{\R^N} V > 0\). 
In the \emph{critical frequency case} \(\inf_{\R^N} V = 0\), solutions concentrating around nondegenerate critical points \cite{AmbrosettiMalchiodiRuiz2006} and around local minima have been obtained \cites{AmbrosettiFelliMalchiodi2005,BonheureVanSchaftingen2008} provided that the potential \(V\) does not decay too fast at infinity.
In the case of local minima, the variational method has been adapted to construct solutions concentrating around a local minimum with a fast decay potential \(V\) --- including a compactly supported potential \citelist{\cite{MorozVanSchaftingen2009}\cite{MorozVanSchaftingen2010}\cite{YinZhang2009}\cite{BaDengPeng2010}}.

The goal of this work is to establish by a variational method the existence of solutions concentrating around local maxima for fast-decaying potentials. Since any potential that decays at infinity has a global maximum, this shows the existence of solutions for a quite general class of potentials. We also think that this problem is a good test of the robustness and flexibility of the variational methods for solutions concentrating around local maxima and of the penalization method for fast decay potentials.

\medbreak

Our main result is the following

\begin{theorem}
\label{theoremMainSmooth}
Let \(N \geq 1\), \(p > 1\) such that \( \frac{1}{p} > \frac{N - 2}{N + 2}\) and \(V \in C^N(\R^{N}, \R^+)\), \(V \not\equiv 0\) be a nonnegative potential.
If 
\(
 \lim_{\abs{x}\to\infty} V(x) = 0 
\)
and either \(\frac{1}{p} < \frac{N - 2}{N}\) or \(\liminf_{\abs{x} \to \infty} V (x) \abs{x}^2 > 0 \), then, for \(\varepsilon>0\) small enough, the problem~\eqref{problemNLSE}
has a family of positive solution that concentrates around a global maximum of \(V\). 
\end{theorem} 

A typical new potential \(V\) for which this result applies is given by \(V(x) = \frac{1}{1+\abs{x}^4}\) for \(x \in \R^n\). The assumption on \(p\) is optimal, since when \(\frac{1}{p} \ge \frac{N-2}{N}\) and \(V\) is compactly supported, \eqref{problemNLSE} does not have any solution. 
Indeed, such a solution would be positive and satisfy \(-\Delta u = u^p\) in \(\R^N \setminus \supp V\) and that would imply \(u = 0\) on \(\R^N \setminus \supp V\) \cite{BV}. 
Theorem~\ref{theoremMainSmooth} follows from the following result:

\begin{theorem}\label{theoremMainLambda}
Let \(N \geq 1\), \(p > 1\) such that \( \frac{1}{p} > \frac{N - 2}{N + 2}\) and \(V \in C(\R^N,\R^+)\). 
Assume that there exists a smooth bounded open set \(\Lambda \subset \R^N\) such that 
\begin{equation*}
 \sup_{\Lambda} V > \inf_{\Lambda} V = \sup_{\partial\Lambda} V 
\end{equation*}
and
\begin{equation*}
  \sup_{\Lambda} V^{\frac{p+1}{p-1}-\frac{N}{2}} 
 < 2 \inf_{\Lambda} V^{\frac{p+1}{p-1}-\frac{N}{2}}\:.
\end{equation*}
If either \(\frac{1}{p} < \frac{N - 2}{N}\) or \(\liminf_{\abs{x} \to \infty} V (x) \abs{x}^2 > 0 \), then for \(\varepsilon > 0\) small enough, problem~\eqref{problemNLSE} possesses a positive weak
solution \(u_{\varepsilon} \in H^1_\mathrm{loc}(\R^N)\) such that
\(u_{\varepsilon}\) achieves its maximum at \(x_{\varepsilon} \in \Lambda\),
\begin{equation*}
 \liminf_{\varepsilon\to 0} u_{\varepsilon}(x_{\varepsilon}) > 0
\end{equation*}
and
\begin{equation*}
 \liminf_{\varepsilon\to 0} \dist(x_\varepsilon, \R^N \setminus \Lambda) > 0\:. 
\end{equation*}
\end{theorem}

The first assumption on \(V\) implies that
\[
  \sup_{\partial\Lambda} V =  \inf_{\partial\Lambda} V\:,
\]
that is, \(\partial \Lambda\) is a level line of \(V\).

Theorem~\ref{theoremMainSmooth} follows from Theorem~\ref{theoremMainLambda} by taking \(\Lambda_\delta = \{x \in \R^N \st V(x) > \sup V - \delta \}\) for \(\delta > 0\). By Sard's lemma, the set \(\Lambda_\delta\) is smooth for almost every \(\delta > 0\). One applies then Theorem~\ref{theoremMainLambda} and a diagonal argument.

Our method of proof is based on the penalization method \citelist{\cite{delPinoFelmer1996}\cite{delPinoFelmer1997}} adapted to decaying potentials \citelist{\cite{BonheureVanSchaftingen2008}\cite{MorozVanSchaftingen2010}}.
However the decay of the potentials requires us to take some extra care at several steps, especially when lower bounds on the energy of solutions are needed.

\medbreak

This paper is organized as follows. We first introduce a penalized problem (section~\ref{sectionPenalization}) and recall some properties of the associated limiting problem (section~\ref{sectionLimiting}). We then study the asymptotic behaviour of families of critical points (section~\ref{sectAsymptoticsCritical}) and minimizers (section~\ref{sectAsymptoticsMinimizers}) of the energy functional associated to \eqref{problemNLSE}. This allows us to define a minimax level and prove the existence of a solution to the penalized problem in section~\ref{sectionMinimax}. Finally in section~\ref{sectionOriginalProblem} we use the asymptotics and some comparison argument to show that when \(\varepsilon\) is small, our solutions of the penalized problem solve the original problem. Whereas the proof is written for \(N \ge 3\), we highlight in section~\ref{sectionSlow} how the proof can be adapted to the case \(N \le 2\).

\section{The penalized problem}
\label{sectionPenalization}
       
Following M.\thinspace del Pino and P.\thinspace Felmer \cite{delPinoFelmer1997}, we introduce a penalized problem. 
D.\thinspace Bonheure and J.\thinspace Van~Schaftingen \citelist{\cite{BonheureVanSchaftingen2008}\cite{BonheureVanSchaftingen2006}} have introduced a penalized problem for decaying potentials. The penalization for fast decay potentials is due to  V.\thinspace Moroz and J.\thinspace Van~Schaftingen \citelist{\cite{MorozVanSchaftingen2009}\cite{MorozVanSchaftingen2010}}. It was used by D.\thinspace Bonheure together with the authors to study solutions concentrating around spheres \cite{DCBonheureVanSchaftingen2008}. Another penalized problem was defined by Yin Huicheng and Zhang Pingzheng \cite{YinZhang2009} (see also Fei Mingwen and Yin Huicheng \cite{FY} and Ba Na, Deng Yinbin and Peng Shuangjie \cite{BaDengPeng2010}).
       
\subsection{The penalization potential}
Recall that \(\Lambda\) is a bounded domain.
Let  \(x_0 \in \Lambda\) and \(\rho > 0\) be such that
\(\overline{B(x_0,\rho)} \subset \Lambda\), and let
\(\chi_{\Lambda}\) denote the characteristic function of the set \(\Lambda\). For \(N \ge 3\), the penalization
potential \(H : \R^N \to \R\) is defined by
\begin{equation*}
 H(x) := \bigl(1-\chi_{\Lambda}(x)\bigr) \frac{(N-2)^2}{4\abs{x-x_0}^2} \biggl(\frac{\log \frac{\rho}{\rho_0}}{\log \frac{\abs{x-x_0}}{\rho_0}
}\biggr)^{1+\beta}
\end{equation*}
for some fixed \(\beta > 0 \) and \(\rho_0 \in (0, \rho)\).
       
Let us recall that the operator \(-\Delta - H\) satisfies a positivity principle \cite{MorozVanSchaftingen2010}*{Lemma 3.1}.
       
\begin{lemma}
\label{lemmaPositivity}
For every \(u \in C^\infty_c (\R^N)\), 
\begin{equation*}
 \int_{\R^N} \bigl( \abs{\nabla u}^2 - H \abs{u}^2 \bigr) \geq 0\:.
\end{equation*}
\end{lemma}
\begin{proof}
Since \(N \ge 3\), this follows from the classical Hardy inequality since for every \(x \in \R^N \setminus B(x_0, \rho)\), 
\[     
  H(x) \le \frac{(N-2)^2}{4\abs{x-x_0}^2}\:.\qedhere
\]     
\end{proof}
       
\subsection{The penalized nonlinearity}
Fix \(\mu \in (0,1)\). The penalized nonlinearity \(g_{\varepsilon}: \R^N \times \R \rightarrow \R\) is defined for \(x \in \R^N\) and \(s \in \R\) by
\begin{equation*}
  g_{\varepsilon}(x,s) := \chi_{\Lambda}(x) s_+{}^p + \bigl( 1-\chi_{\Lambda}(x) \bigr)
\min\bigl(\mu \bigl(\varepsilon^2 H(x) + V(x)\bigr), \abs{s}^{p-1} \bigr) s_+\:.
\end{equation*}
Also set  \(G_{\varepsilon}(x,s) := \int_0^s g_{\varepsilon}(x,\sigma)\, d\sigma\). The function \(g_{\varepsilon}\) is a Carath\'eodory function with the following properties :
\begin{itemize}
 \item[(\(g_1\))] \(g_{\varepsilon}(x,s) = o(s),\) as \(s \rightarrow 0\), uniformly on compact subsets of
\(\R^N\).
 \item[(\(g_2\))] for every \(x \in \R^N\) and \(s \in \R\),
\[
   g_{\varepsilon}(x,s) \le (s)_+^p,
\]
if moreover \(x \in \R^N \setminus \Lambda\), then 
\[
 g_{\varepsilon}(x,s) \leq \mu\bigl(\varepsilon^2 H(x)  + V(x) \bigr)s_+
\]
 \item[(\(g_3\))] for every \(s \in \R\), if \(x \in \Lambda\),
 \[
   (p+1) G_{\varepsilon}(x,s) \leq g_{\varepsilon}(x,s)s\:,
\]
and if \(x \not \in \Lambda\),
\[
   2 G_{\varepsilon}(x,s) \leq g_{\varepsilon}(x,s)s \:,
\]
\item[(\(g_4\))] the function
 \begin{equation*}
 t \in (0, \infty) \mapsto \frac{g_{\varepsilon}(x, ts)s}{t}
 \end{equation*}
 is nondecreasing for all \(x \in \R^N\) and \(s \in \R\).
\end{itemize}
       
\subsection{The penalized functional}
       
The Hilbert space naturally associated to the linear part of our equation is the 
weighted Sobolev space \(H^1_{V}(\R^N)\),
which is the closure of \(C^\infty_c(\R^n)\) under any of the equivalent norms
\begin{equation*}
 \norm{u}_{\varepsilon}^2 := \int_{\R^N} \bigl( \varepsilon^2 \abs{\nabla u}^2 + V \abs{u}^2 \bigr)
\end{equation*} 
defined for \(\varepsilon > 0\).
       
We look for a solution \(u \in H^1_V (\R^N)\) of the penalized equation
\begin{equation}
\label{problemPNLSE}
\tag{\protect{$\mathcal{Q}_\varepsilon$}}
  - \varepsilon^2 \Delta u(x) + V(x)\,u(x) = g_{\varepsilon}\bigl(x,u(x)\bigr) \qquad \text{for \(x \in \R^N\)}.	
\end{equation}
The associated functional is 
\begin{equation*}
 \mathcal{J}_{\varepsilon} : H^1_V \rightarrow \R : \mathcal{J}_{\varepsilon}(u) := \frac{1}{2} \int_{\R^N} \bigl( \varepsilon^2
\abs{\nabla u}^2 + V\abs{u}^2 \bigr) - \int_{\R^N} G_{\varepsilon}\bigl(x,u(x)\bigr)\: dx\;.
\end{equation*}
It is standard that \(\mathcal{J}_{\varepsilon}\) is well-defined and continuously differentiable and that its critical points are weak solutions of the penalized equation \eqref{problemPNLSE}.
       
\subsection{The Nehari manifold}
The Nehari manifold associated to the functional 
\(\mathcal{J}_{\varepsilon}\) is defined by
\begin{equation*}
 \mathcal{N}_{\varepsilon} := \bigl\{ u \in H^1_V(\R^N) \setminus \{0\}\st \dualprod{ \mathcal{J}_{\varepsilon}'(u)}{u}=0
\bigr\}\;.
\end{equation*} 
It is well-known that \(u \in H^1_V(\R^N)\setminus \{0\}\) is a critical point of \(\mathcal{J}_{\varepsilon}\) if and only if \(u \in \mathcal{N}_{\varepsilon}\) and \(u\) is a critical point of \(\mathcal{J}_{\varepsilon}\) restricted to \(\mathcal{N}_{\varepsilon}\). 
       
We point out that \(\mathcal{N}_{\varepsilon}\) is bounded away from \(0\). We first have an integral estimate.

\begin{lemma}
\label{lemENehariEstimate1}
Let \(\varepsilon > 0\) and
\(     
 u \in \mathcal{N}_\varepsilon.
\)     
Then   
\[     
   \int_{\Lambda} (u)_+^{p+1} 
  \ge (1-\mu) \int_{\R^N} \varepsilon^2 \abs{\nabla u}^2 +V \abs{u}^2.
\]     
\end{lemma}
\begin{proof}
By \((g_2)\), one has
\[     
\begin{split}
   \int_{\R^N} \bigl(\varepsilon^2 \abs{\nabla u}^2 +V \abs{u}^2\bigr)
&=\int_{\R^N} g_\varepsilon(x, u(x)) u(x)\:dx\:\\
&\le \int_{\Lambda} \abs{u}^{p+1} +\mu \int_{\R^N \setminus \Lambda} (V+\varepsilon^2 H)\abs{u}^2 .
\end{split}
\]
We deduce from Lemma~\ref{lemmaPositivity} that
\[
 (1-\mu) \int_{\R^N} \varepsilon^2 \abs{\nabla u}^2 +V \abs{u}^2
\le \int_{\Lambda} \abs{u}^{p+1}.\qedhere
\]
\end{proof}

\begin{lemma}
\label{lemENehariEstimate2}
Let \(\varepsilon > 0\) and
\(     
 u \in \mathcal{N}_\varepsilon.
\)     
Then   
\[     
  \int_{\Lambda} \varepsilon^2 \abs{\nabla u}^2 +V \abs{u}^2
  \ge c\varepsilon^{N}
\]     
where \(c > 0\) is independent of \(\varepsilon\) and \(u\).
\end{lemma}
\begin{proof}
Since  \(\inf_\Lambda V > 0\), the Sobolev and H\"older inequalities imply that
\[
\begin{split}
  \int_{\Lambda} \abs{u}^{p+1} &\le C \Bigl( \int_{\R^N} \abs{\nabla u}^2 \Bigr)^{\frac{p-1}{4} N }
\Bigl( \int_{\Lambda} \abs{u}^2\Bigr)^{\frac{p+1}{2}-\frac{p-1}{4}N}\\
&\le \frac{C}{\varepsilon^{\frac{p-1}{2} N}} \Bigl( \int_{\Lambda} \varepsilon^2 \abs{\nabla u}^2 + V \abs{u}^2 \Bigr)^\frac{p+1}{2}.
\end{split}
\]
The conclusion follows from Lemma~\ref{lemENehariEstimate1}.
\end{proof}

We also have a uniform lower estimate on the maximum.
\begin{lemma}\label{lem:1}
Let \(\varepsilon > 0\) and
\(u \in \mathcal{N}_\varepsilon\). Then
\begin{equation*}
 \sup_{\Lambda}\frac{u_+^{p-1}}{V} \ge 1\:.
\end{equation*}
\end{lemma}
This was proved for solutions of \eqref{problemPNLSE} by V.\thinspace Moroz and J.\thinspace Van Schaftingen \cite{MorozVanSchaftingen2010}*{Lemma 4.2} (see also~\cite{BonheureVanSchaftingen2008}*{Lemma 17}).

\begin{proof}
One has by \((g_2)\), 
\[
\begin{split}
\int_{\R^N} \varepsilon^2 \abs{\nabla u}^2 +V \abs{u}^2 &\le \int_{\Lambda} (u)_+^{p+1}
+\mu \int_{\R^N \setminus \Lambda} (V + \varepsilon^2 H)\abs{u}^2 \\
&\le \sup_{\Lambda}\frac{u_+^{p-1}}{V} \int_{\Lambda} V \abs{u}^{2}
+\mu \int_{\R^N \setminus \Lambda} (V + \varepsilon^2 H)\abs{u}^2\:,
\end{split}
\]
and thus by Lemma~\ref{lemmaPositivity},
\[
 \sup_{\Lambda}\frac{u_+^{p-1}}{V} \int_{\Lambda} V \abs{u}^{2} \ge \int_{\Lambda} V \abs{u}^{2}\:.
\]
By Lemma~\ref{lemENehariEstimate2}, \(\int_{\Lambda} V \abs{u}^{2} > 0\), and the conclusion follows.
\end{proof}

We also note the following coercivity estimate.
\begin{lemma}\label{lem:bounded}
For every \(\varepsilon > 0\) and \(u \in \mathcal{N}_{\varepsilon}\),
\[
  \Bigl(\frac{1}{2}-\frac{1}{p+1}\Bigr)(1-\mu) \int_{\R^N} \bigl( \varepsilon^2 \abs{\nabla u}^2 + V
\abs{u}^2\bigr) \leq \mathcal{J}_{\varepsilon}(u)\:.
\]
\end{lemma}
\begin{proof}
Since \(u \in \mathcal{N}_\varepsilon\), one has 
\begin{multline*}
  \mathcal{J}_{\varepsilon}(u)=\Bigl(\frac{1}{2}-\frac{1}{p+1}\Bigr) \int_{\R^N} \bigl( \varepsilon^2 \abs{\nabla u}^2 + V \abs{u}^2 \bigr)\\
 + \frac{1}{p+1} \int_{\R^N}  g_\varepsilon\bigl(x, u(x)\bigr) u(x) - (p+1) G_\varepsilon\bigl(x, u(x)\bigr)\,dx.
\end{multline*}
In view of \((g_3)\) and \((g_2)\), 
\begin{multline*}
   \frac{1}{p+1} \int_{\R^N}  g_\varepsilon\bigl(x, u(x)\bigr) u(x) - {(p+1)} G_\varepsilon(x, u(x))\,dx\\
\geq -\frac{p-1}{p+1} \int_{\R^N \setminus \Lambda} G_\varepsilon\bigl(x, u(x)\bigr)\,dx \\
\geq - \Bigl(\frac{1}{2}-\frac{1}{p+1}\Bigr) \mu \int_{\R^N \setminus \Lambda} ( \varepsilon^2 H + V )\abs{u}^2.
\end{multline*}
Thanks to Lemma~\ref{lemmaPositivity}, we reach the conclusion.
\end{proof}

\subsection{The Palais-Smale condition}
For every \(\varepsilon > 0\), the functional \(\mathcal{J}_{\varepsilon}\) satisfies the Palais-Smale compactness condition: 

\begin{lemma}\label{lem:PalaisSmale}
For every \(\varepsilon > 0\), if \( (u_n)_{n \in \N} \) is a sequence such that \((\mathcal{J}_{\varepsilon}(u_n))_{n \in \N}\) converges and \((\mathcal{J}_{\varepsilon}'(u_n))_{n \in \N} \) converges to \(0\) in \( (H^1_V(\R^N))' \), then, up to a subsequence, \( (u_n)_{n \in \N} \) converges in \(H^1_V(\R^N)\).
\end{lemma}

The proof of Lemma~\ref{lem:PalaisSmale} is a combination of the arguments for the penalization without \(H\) \cite{BonheureVanSchaftingen2008}*{Lemma 6} and without \(V\) \cite{MorozVanSchaftingen2010}*{Lemma 3.5} whose main lines originate in the proof  for nondecaying potentials \cite{delPinoFelmer1996}*{Lemma 1.1}. It was already proved with the present penalization for the functional restricted to a subspace of symmetric functions \cite{DCBonheureVanSchaftingen2008}.

\subsection{Minimizers on the Nehari manifold}

\begin{proposition}\label{prop:Minimizer}
For every \(\varepsilon > 0\), there exists \(u \in \mathcal{N}_\varepsilon\) such that
\[
 \mathcal{J}_{\varepsilon}(u) = \inf_{\mathcal{N}_\varepsilon}\mathcal{J}_{\varepsilon}\:.
\]
\end{proposition}

Proposition~\ref{prop:Minimizer} was proved for the penalization for nondecaying potentials \cite{delPinoFelmer1996}*{Lemma 2.1}, the penalization without \(H\) \cite{BonheureVanSchaftingen2008}*{Proposition 9} the penalization without \(V\) \cite{BonheureVanSchaftingen2008}*{Proposition 3.7} and the present penalization under symmetry constraints \cite{DCBonheureVanSchaftingen2008}.

\begin{proof}[Proof of Proposition~\ref{prop:Minimizer}]
The proof is standard: by \( (g_4)\) one has the equality \cite{Rabinowitz1992}*{Proposition 3.11}
\[
\inf_{\mathcal{N}_\varepsilon}\mathcal{J}_{\varepsilon}= \inf_{\substack{u \in H^1_V(\R^N) \\ u_+ \vert_\Lambda \ne 0}} \sup_{t > 0} \mathcal{J}_{\varepsilon}(tu)= \inf_{\substack{\gamma \in C([0, 1], H^1_V(\R^N)) \\ \gamma(0)=0 \\ \mathcal{J}_{\varepsilon}(\gamma(1)) < 0}} \sup_{t \in [0, 1]} \mathcal{J}_{\varepsilon}(\gamma(t))\:;
\]
by Lemmas \ref{lemENehariEstimate2} and \ref{lem:bounded}, \(\mathcal{J}_{\varepsilon}\) is bounded away from \(0\) on \(\mathcal{N}_\varepsilon\). Since \(\mathcal{J}_{\varepsilon} \) satisfies the Palais-Smale compactness condition by Lemma~\ref{lem:PalaisSmale}, the existence of \(u\) follows.
\end{proof}

\section{Limiting problems}
\label{sectionLimiting}

\subsection{The limit problem}
For \(\nu > 0\) let \(U_{\nu}\) be the unique positive solution of the problem
\begin{equation}\label{problemLimit}
\left\{ 
\begin{aligned}
 -\Delta u + \nu u &= u^p & &\text{in \(\R^{N}\)}, \\
 u &> 0, \\
 u(0) &= \max_{\R^N} u.
\end{aligned}\right.
\end{equation}
The function \(U_{\nu}\) is radial around the origin \cite{Kwong}. The functional associated to \eqref{problemLimit}
is \(\mathcal{I}_{\nu} : H^1(\R^N) \to \R\) defined for \(u \in H^1(\R^N)\) by 
\begin{equation*}
 \mathcal{I}_{\nu}(u) := \frac{1}{2} \int_{\R^N} \bigl( \abs{\nabla u}^2 + \nu \abs{u}^2 \bigr) - \frac{1}{p+1} \int_{\R^N} u_+^{p+1}.
\end{equation*}
One has the variational characterization
\[
    b_{\nu} := \inf_{\mathcal{M}_\nu} \mathcal{I}_\nu = \mathcal{I}_{\nu}(U_{\nu})
\]
where
\[
 \mathcal{M}_\nu = \Bigl\lbrace u \in H^1(\R^N)\setminus \{0\} \st \int_{\R^N} \bigl(\abs{\nabla u}^2 + \nu \abs{u}^2 \bigr) = \int_{\R^N} u_+^{p+1} \Bigr\rbrace\:.
\]
We also set
\begin{equation}\label{estim:by}
 \mathcal{C}(y) := b_{V(y)} = \frac{S_{p+1}^r}{r} V(y)^{\frac{p+1}{p-1}-\frac{N}{2}},
\end{equation}
where \(\frac{1}{r}=\frac{1}{2}-\frac{1}{p+1}\) and
\begin{equation*}
 S_{p+1}^2 := \inf \Bigl\lbrace \int_{\R^N} \bigl(\abs{\nabla u}^2 + \abs{u}^2\bigr) \st u \in
H^1(\R^N)\ \text{and}\ \int_{\R^N} u_+^{p+1} = 1
\Bigr\rbrace\:.
\end{equation*}

We also recall the following classical result
\begin{lemma}
\label{lemmaConvergenceMinimizingLimiting}
Let \(\nu > 0\) and \( (v_n)_{n\in \N}\) be a sequence in \(\mathcal{M}_\nu \subset H^1(\R^N)\). If 
\[\lim_{n \to \infty} \mathcal{I}_\nu(v_n) = b_{\nu}, \]
then there exists a sequence of points \((y_n)_{n\in \N}\) in \(\R^N\) such that \(v_n(\cdot -y_n) \to U_\nu\) in \(H^1(\R^N)\).
\end{lemma}
\begin{proof}
Let \(\mathcal{F}_\nu : H^1(\R^N) \to \R\) be defined for \(v \in H^1(\R^N)\) by
\[
\mathcal{F}_\nu (v) = \int_{\R^N} \abs{\nabla v}^2 + \nu \abs{v}^2 - \int_{\R^N} v_+^{p+1}
\]
By a standard application of the Ekeland variational principle on the manifold \(\mathcal{M}_\nu\)  (see for
example \cite{MW}*{Theorem 4.1}), there exist sequences \((\tilde{v}_n)_{n \in \N} \subset \mathcal{M}_\nu\) and \((\lambda_n)_{n \in \N} \subset \R\) such that \(\mathcal{I}_{\nu}(\tilde{v}_n)\to b_{\nu}\), 
\(\mathcal{I}_{\nu}'(\tilde{v}_n)+\lambda_n \mathcal{F}_\nu'(\tilde{v}_n)\to 0\) in \(H^{-1}(\R^N)\) and \(v_n - \tilde{v}_n \to 0\) in \(H^1(\R^N)\) as \(n \to \infty\). 
The sequence \((\tilde{v}_n)_{n \in \N}\) is a Palais-Smale sequence for the unconstrained functional \(\mathcal{I}_{\nu}\), that is \(\mathcal{I}_{\nu}'(\tilde{v}_n) \to 0\).
Indeed one has 
\begin{align*}
   \lambda_n (p-1)\int_{\R^N} \abs{\nabla v_n}^2 + \abs{v_n}^2 &= -\lambda_n \dualprod{\mathcal{F}_\nu'(\tilde{v}_n)}{\Tilde{v}_n} \\
   &= \dualprod{\mathcal{I}_{\nu}'(\tilde{v}_n)}{\tilde{v}_n}+ o(\norm{\Tilde{v}_n}_{H^1})=o(\norm{\Tilde{v}_n}_{H^1})\:,
\end{align*}
as \(n \to \infty\).
Since there exists a constant \(c>0\) such that \(\norm{v}_{H^1(\R^N)} \geq c\) for every \(v \in \mathcal{M}_\nu\) we deduce that 
\(\lim_{n\to \infty} \lambda_n = 0\).

We compute that
\begin{equation*}
 2 \mathcal{I}_{\nu}(\tilde{v}_n) - \dualprod{\mathcal{I}_{\nu}'(\tilde{v}_n)}{\tilde{v}_n} = \Bigl( 1 - \frac{2}{p+1} \Bigr) \int_{\R^N} (\tilde{v}_n)_+^{p+1}\: dx \to 2 b_{\nu}\:.
\end{equation*}
Hence, 
\begin{equation*}
 \liminf_{n \to \infty} \int_{\R^N} (\tilde{v}_n)_+^{p+1}  > 0\:.
\end{equation*}
Since \((\tilde{v}_n)_{n \in \N}\) is bounded in \(H^1(\R^N)\), we deduce from \cite{Lions1984}*{Part 2, Lemma I.1} (see also \cite{Willem1996}*{Lemma 1.21}) that
\begin{equation*}
 \int_{\R^N} (\tilde{v}_n)_+^{p + 1} \le C \Bigl(\int_{\R^N} \abs{\nabla \tilde{v}_n}^2 + \abs{\tilde{v}_n}^2 \Bigr)\sup_{z \in \R^N} \Bigl(\int_{B(z,1)} (\tilde{v}_n)^{p + 1} \Bigr)^\frac{p - 1}{p + 1}\:.
\end{equation*}
Consequently, there exists a sequence \((y_n)_{n \in \N} \subset \R^N\) such that, if we set \(\Bar{v}_n := \tilde{v}_n(\cdot-y_n)\), we have
\begin{equation}\label{ineqConvergenceMinimizingLimitingBalls}
 \liminf_{n \to \infty} \int_{B(0,1)} (\Bar{v}_n)_+^{p + 1}  > 0.
\end{equation}
Since \((\Bar{v}_n)_{n \in \N}\) is bounded in \(H^1(\R^N)\) and \(\frac{1}{p} > \frac{N - 2}{N}\), we can assume that \(\Bar{v}_n \rightharpoonup \Bar{v}\) in \(H^1(\R^N)\), 
\(\Bar{v}_n \to
\Bar{v}\) in \(L^{p + 1}_{\text{loc}}(\R^N)\) and \(\Bar{v}_n \to \Bar{v}\) almost everywhere. 
By \eqref{ineqConvergenceMinimizingLimitingBalls}, \(\Bar{v} \not\equiv 0\). For all \(v \in H^1(\R^N)\), we have \(\dualprod{\mathcal{I}_{\nu}'(\Bar{v}_n)}{v} \to 0\)
because \(\Bar{v}_n\) is a Palais-Smale sequence, and \(\dualprod{\mathcal{I}_{\nu}'(\Bar{v}_n)}{v} \to \dualprod{\mathcal{I}_{\nu}'(\Bar{v})}{v}\) because \(\Bar{v}_n \rightharpoonup \Bar{v}\). We
conclude that \(\dualprod{\mathcal{I}_{\nu}'(\Bar{v})}{v} = 0\) and so \(\Bar{v}\) is a solution of \eqref{problemLimit}. We compute that
\begin{equation*}
 \frac{p-1}{p+1} \int_{\R^N} (\Bar{v}_n)_+^{p+1} = 2 \mathcal{I}_{\nu}(\Bar{v}_n) - \dualprod{\mathcal{I}_{\nu}'(\Bar{v}_n)}{\Bar{v}_n} \to 2 b_{\nu}
\end{equation*}
and
\begin{equation*}
 \frac{p-1}{p+1} \int_{\R^N} (\Bar{v})_+^{p+1} = 2 \mathcal{I}_{\nu}(\Bar{v}) = 2 b_{\nu}.
\end{equation*}
Therefore \(\norm{(\Bar{v}_n)_+}_{L^{p+1}(\R^N)} \to \norm{(\Bar{v})_+}_{L^{p+1}(\R^N)}\). We infer that \(\Bar{v}_n \to \Bar{v}\) in \(L^{p+1}(\R^N)\). Finally we
can write
\begin{equation*}
\begin{split}
 \norm{\Bar{v}_n-\Bar{v}}_{H^1(\R^N)}^2 &= \dualprod{\mathcal{I}_{\nu}'(\Bar{v}_n)-\mathcal{I}_{\nu}'(\Bar{v})}{\Bar{v}_n-\Bar{v}} \\ &\qquad+ \int_{\R^N} \left(
(\Bar{v}_n)_+^{p} - (\Bar{v})_+^{p} \right) \left( \Bar{v}_n-\Bar{v} \right)\:.
\end{split}
\end{equation*}
Since \(\mathcal{I}_{\nu}'(\Bar{v}) = 0\), \(\mathcal{I}_{\nu}'(\Bar{v}_n) \to 0\) as \(n \to \infty\) and the last term goes to \(0\) by H\"older's inequality, we conclude that
\(\Bar{v}_n \to \Bar{v}\) in \(H^1(\R^N)\). The conclusion follows.
\end{proof}

\subsection{Penalized limit problems}

The two following lemmas will provide information about the limit of sequences of rescaled solutions. The first lemma is due to M.\thinspace del Pino and P.\thinspace Felmer \cite{delPinoFelmer1998}*{Lemma 2.3}.
Let \(\R^N_+ := \left\lbrace x \in \R^N \st x_N > 0 \right\rbrace\).

\begin{lemma}\label{lem:plimpenal1}
Let \(\nu \geq 0\) and \(\mu \in [0, \nu]\). If \(u \in H^1(\R^N)\) is a
solution of 
\begin{equation*}
\left\{ \begin{aligned} 
          - \Delta u + \nu u &= u_+^p & & \text{in \(\R^N_+\)}, \\
	  - \Delta u + \nu u &= \min \bigl( \mu, \abs{u}^{p-1} \bigr)u_+ & &\text{in \(\R^N_-\)},
         \end{aligned}\right.
\end{equation*}
then \(\abs{u}^{p-1} \leq \mu\) in \(\R^N_-\).
\end{lemma}
\begin{proof}
We follow the argument of M.\thinspace del Pino and P.\thinspace Felmer \cite{delPinoFelmer1998}*{Lemma 2.3}. By elliptic regularity, \(u \in H^2(\R^N) \cap C^1(\R^N)\). 
Thus we can use \(\partial_N u\) as a test function in the equation.
Writing \(g(s) := \min ( \mu, \abs{s}^{p-1})s_+\) and \(G(s) := \int_0^s g(\sigma)\; d\sigma\), we obtain
\begin{equation*}
 \frac{1}{2} \int_{\R^N} \partial_N \bigl( \abs{\nabla u}^2 + \nu \abs{u}^2 \bigr) = \frac{1}{p+1} \int_{\R^N_+}
 \partial_N \bigl( u_+^{p+1} \bigr) + \int_{\R^N_-} \partial_N ( G \circ u)\:.
\end{equation*}
This reduces to
\begin{equation*}
 \int_{\R^{N-1}} \Bigl( G\bigl( u(x', 0) \bigr) - \frac{1}{p+1} \bigl(u(x',0)\bigr)_+^{p+1} \Bigr)\; dx' = 0.
\end{equation*}
Since \(G(u) \leq \frac{u_+^{p+1}}{p+1}\) on \(\R^N\), we have \(G\left( u(x', 0) \right) = 
\frac{1}{p+1} (u(x',0))_+^{p+1}\) for all \(x' \in \R^{N-1}\) and hence,  for every \(x'\in \R^{N-1}\),
\(
 u(x',0) \leq \mu^{\frac{1}{p-1}}\:.
\)
One has on \(\R^N_-\)
\[
 -\Delta u + (\nu-\mu)u \le 0\:. 
\]
Since \(\nu \ge \mu\), we deduce by the maximum principle that \(u \leq \mu^{\frac{1}{p-1}}\) in \(\R^N_-\). 
\end{proof}

The second lemma is an application of the maximum principle.

\begin{lemma}\label{lem:plimpenal2}
Let \(\nu \geq 0\) and \(\mu \in [0, \nu]\). If \(u \in H^1(\R^N)\), \(u\geq 0\), is a solution of
\begin{equation*}
  - \Delta u + \nu u = \min \bigl( \mu, u^{p-1} \bigr)u_+ \qquad
\text{in \(\R^N\)},
\end{equation*}
then \(u \equiv 0\).
\end{lemma}
\begin{proof}
If \(u\) is a solution, we have
\[
- \Delta u + (\nu - \mu) u \leq 0 \qquad \text{in \(\R^N\)}\:.
\]
Taking \(u\) as a test function, we obtain
\[
\int_{\R^N} \abs{\nabla u}^2 +  (\nu - \mu) \abs{u}^2 \leq 0\:.
\]
Since \(\nu - \mu \geq 0\), this implies that \(u \equiv 0\).
\end{proof}

\section{Asymptotics of families of critical points}
\label{sectAsymptoticsCritical}

In this section we refine the 
asymptotic analysis in \cite{BonheureVanSchaftingen2008}*{Section 5} in order to obtain an estimate of the energy of a critical point \(u_{\varepsilon}\) of
\(\mathcal{J}_{\varepsilon}\)
depending on the number and the location of its local maxima. The corresponding lower estimate was proved in \cite{BonheureVanSchaftingen2008}.

\subsection{Asymptotics on small balls}

The next lemma states that the sequences of rescaled solutions converge in \(C^1_{\text{loc}}(\R^N)\)
to a function in \(H^1(\R^N)\). 

\begin{lemma}
\label{lemuepsrescaledcompact}
Let \((\varepsilon_n)_{n \in \N} \) be a sequence in \(\R^+\) such that
\(\varepsilon_n \to 0\) as \(n \to \infty\), let \((u_{n})_{n \in \N}\) be a sequence of solutions of \(\mathcal{Q}_{\varepsilon_n}\) such that
\[
  \liminf_{n \to \infty} \varepsilon_n^{-N} \mathcal{J}_{\varepsilon_n}(u_n) < \infty
\]
and let \((x_n)_{n \in \N}\) be a  sequence in \(\R^N\) such that \(x_n \to \Bar{x}\) as \(n \to \infty\).
Denote by \((v_n)_{n \in \N}\) the sequence
defined by \(v_n(x) = u_{\varepsilon_n}(x_n + \varepsilon_n x)\). If \(V(\Bar{x}) > 0\),
then there exists \(v \in H^1(\R^N)\) such that, up to a subsequence,
\(
 v_n \to v
\)
in \(C^1_{\mathrm{loc}}(\R^{N})\) 
as \(n \to \infty\).
\end{lemma}

This lemma was proved for minimal energy solutions in \cite{BonheureVanSchaftingen2008}*{Lemma 13}. We sketch here the argument in order to highlight that the proof only depends on the fact that \(u_n\) is a solution that satisfies an energy bound.

\begin{proof}
Take \(\varphi \in C^\infty_c(\R^N) \) such that \(\varphi \equiv 1 \) on \(B(0,1)\). 
Set for \(R > 0\) and \(x \in \R^N\) \(\varphi_R(x)=\varphi(\frac{x}{R})\). The sequence
\( (\varphi_R v_n)_{n \in \N}\) is bounded in \(H^1(\R^N)\) for every \(R > 0\). By a diagonal argument,  there exists \(v \in H^1_\mathrm{loc}(\R^N) \) such that \(v_n \to v\) weakly in \(H^1_\mathrm{loc}(\R^N)\) along a subsequence. 

Now note that for every \(R > 0\),
\[
  \int_{B(0,R)} \abs{\nabla v}^2 \leq  \liminf_{n \to \infty} \int_{B(0,R)} \abs{\nabla v_n}^2 \leq  \liminf_{n \to \infty} \int_{\R^N} \abs{\nabla v_n}^2
\]
and
\[
 V(\Bar{x}) \int_{B(0,R)} \abs{v}^2 \leq \liminf_{n \to \infty} \int_{B(0,R)} V \abs{v_n}^2
\leq \liminf_{n \to \infty} \int_{\R^N} V \abs{v_n}^2\:,
\]
so that \(v \in H^1(\R^N)\).

The remainder follows from classical regularity and compactness results.
\end{proof}

\begin{lemma}
\label{lemsmallballs}
Let \((\varepsilon_n)_{n \in \N} \) be a sequence in \(\R^+\) such that
\(\varepsilon_n \to 0\) as \(n \to \infty\), let \((u_{n})_{n \in \N}\) be a sequence of solutions of \(\mathcal{Q}_{\varepsilon_n}\) such that
\[
  \liminf_{n \to \infty} \varepsilon_n^{-N} \mathcal{J}_{\varepsilon_n}(u_n) < \infty
\]
and let \((x_n)_{n \in \N}\) be a  sequence in \(\R^N\) such that \(x_n \to \Bar{x}\) as \(n \to \infty\). If \(V(\Bar{x}) > 0 \) and
\begin{equation*}
 \liminf_{n\to\infty} u_{\varepsilon_n}(x_n) > 0\:,
\end{equation*}
then, \(\Bar{x} \in \Bar{\Lambda}\),
\[
 \limsup_{n \to \infty} \frac{\dist(x_n, \Lambda)}{\varepsilon_n} < \infty\:,
\]
\[
 \lim_{R\to \infty} \limsup_{n\to \infty} \Bigabs{\varepsilon_n^{-N}  \int_{B(x_n, \varepsilon_n R)} \Bigl( \frac{1}{2}
\bigl( \varepsilon_n^2 \abs{\nabla u_{\varepsilon_n}}^2 + V \abs{u_{\varepsilon_n}}^2 \bigr) -
G_{\varepsilon_n}(.,u_{\varepsilon_n}) \Bigr) -\mathcal{C}(\Bar{x})} = 0\:,
\]
and
\[
  \lim_{R\to \infty} \limsup_{n\to \infty} \varepsilon_n^{-N}  \int_{B(x_n, 2\varepsilon_n R) \setminus B(x_n, \varepsilon_n R)} 
\abs{\nabla u_{\varepsilon_n}}^2 + V \abs{u_{\varepsilon_n}}^2=0\:.
\]
\end{lemma}
\begin{proof}
Set \(v_n(x) := u_{\varepsilon_n}(x_n + \varepsilon_n x)\). By Lemma~\ref{lemuepsrescaledcompact} up to a subsequence, there exists a \(v \in H^1(\R^N)\) such that 
\(v_n \to v\) in \(C^1_{\text{loc}}(\R^N)\). We have
\(v(0) = \lim_{n\to \infty} v_n(0) > 0\) so that \(v \not \equiv 0\).

Let us now prove by contradiction that 
\begin{equation}
\label{limsupdistxneps}
  \limsup_{n \to \infty} \frac{\dist(x_n, \Lambda)}{\varepsilon_n} < \infty\:.
\end{equation}
Up to a subsequence, we can assume that \(\lim_{n \to \infty} \dist(x_n, \Lambda)/\varepsilon_n = \infty\).
Then, since the sequence of characteristic functions \(\chi_n(x) :=
\chi_{\Lambda}(x_n + \varepsilon_n x)\) converges pointwise to \(0\), we have as \(n \to \infty\)
\[
\begin{split}
g_{\varepsilon_n}(x_n + \varepsilon_n \cdot, v_n) & = \min \bigl( \mu \bigl(\varepsilon_n^2 H(x_n + \varepsilon_n \cdot) v_n
+ V(x_n + \varepsilon_n \cdot) \bigr), \abs{v_n}^{p-1} \bigr)(v_n)_+ \\ 
&\qquad \to \min \bigl( \mu V(\Bar{x}) , \abs{v}^{p-1} \bigr)v_+\: ,
\end{split}
\]
in \(L^q_{\mathrm{loc}}(\R^N)\) for \(1 \le q < \frac{2N}{p(N-2)}\),
and thus \(v\) solves the limiting equation
\begin{equation*}
 -\Delta v + V(\Bar{x}) v = \min \bigl(\mu V(\Bar{x}) , \abs{v}^{p-1} \bigr)v_+,
\qquad\text{in \(\R^N\)}.
\end{equation*}
By Lemma \ref{lem:plimpenal2}, \(v \equiv 0\), which is a contradiction. Thus \eqref{limsupdistxneps} holds.

Now, let us assume that
\begin{equation}
\label{limsupdistxnepsc}
  \limsup_{n \to \infty} \frac{\dist(x_n, \R^N \setminus \Lambda)}{\varepsilon_n} = \infty\:.
\end{equation}
Since \(\chi_n(x)\) converges pointwise to \(1\), we have, up to a subsequence, for \(n\) large enough,
\begin{equation*}
g_{\varepsilon_n}(x_n + \varepsilon_n \cdot, v_n) = (v_n)_+^p \to v_+^p\:,
\end{equation*}
in \(L^q_{\mathrm{loc}}(\R^N)\) for \(1 \le q < \frac{2N}{p(N-2)}\).
Hence \(v\) solves the limiting equation
\begin{equation}\label{plim4}
 - \Delta v + V(\Bar{x}) v = v_+^p \qquad \text{in }\ \R^N\:.
\end{equation}

If \eqref{limsupdistxneps} holds but \eqref{limsupdistxnepsc} does not, then
\[
 \limsup_{n \to \infty} \frac{\dist(x_n, \partial \Lambda)}{\varepsilon_n} < \infty\:.
\]
Since \(\Lambda\) is smooth, 
\(
 \chi_n \to \chi_{E},
\)
almost everywhere as \(n \to \infty\), where \(E\) is a half-space. 
By Lemma \ref{lem:plimpenal1}, \(v\) is again a solution of \eqref{plim4}.

In any case, \(v\) is thus a nontrivial solution of \eqref{plim4}. Now we claim that
\begin{multline}\label{eqlimhn}
\lim_{R\to \infty} \lim_{n\to \infty} \varepsilon_n^{-N}  \frac{1}{2}\int_{B(x_n, \varepsilon_n R)} 
\bigl( \varepsilon_n^2 \abs{\nabla u_{\varepsilon_n}}^2 + V \abs{u_{\varepsilon_n}}^2 \bigr) \\ = \frac{1}{2}\int_{\R^N} \abs{\nabla U_{V(\Bar{x})}}^2 + V(\Bar{x}) \abs{U_{V(\Bar{x})}}^2\:.
\end{multline}
For every \(R > 0\), the convergence of \(v_n\) to \(v\) in \(C^1_{\text{loc}}(\R^N)\) implies that
\begin{equation*}
 \lim_{n\to \infty} \varepsilon_n^{-N}  \frac{1}{2}\int_{B(x_n, \varepsilon_n R)} 
\bigl( \varepsilon_n^2 \abs{\nabla u_{\varepsilon_n}}^2 + V \abs{u_{\varepsilon_n}}^2 \bigr) \\ = \frac{1}{2}\int_{B(0, R)} \abs{\nabla v}^2 + V(\Bar{x}) \abs{v}^2\:.
\end{equation*}
Since \(v \in H^1(\R^N)\), and \(v\) and \(U_{V(\Bar{x})}\) are equal up to a translation, we conclude that \eqref{eqlimhn} holds.
The argument for the other limit is similar.
\end{proof}

\subsection{Asymptotics outside small balls}

The solutions decay outside a neighborhood of \(\Lambda\):
\begin{lemma}\label{lemmaAsymptoticsOutsideLambda}
For every open set \(U\) such that \(\Bar{\Lambda} \subset U\), there exists \(C > 0\) such that for every \(\varepsilon > 0\), if \(u \in H^1_V (\R^N)\) is a solution of \(\mathcal{Q}_{\varepsilon}\),
\begin{equation*}
   \int_{\R^N\setminus U} \bigl(\varepsilon^2 \abs{\nabla u}^2 + V \abs{u}^2 \bigr) \le C \varepsilon^2 \int_{\R^N} \bigl(\varepsilon^2 \abs{\nabla u}^2 + V \abs{u}^2 \bigr)\:.
\end{equation*}
\end{lemma}

\begin{proof}
Since \(V\) is continuous and \(\inf_{\Lambda} V > 0\), we can assume without loss of generality that \(\inf_{U} V > 0\).
Take \(\psi \in C^\infty_c(\R^N)\) such that \(\psi = 0\) on \(\Bar{\Lambda}\) and \(\psi=1\) on \(\R^N \setminus U\). By taking \(\psi^2 u_{n}\) as a test function in \eqref{problemPNLSE}, we obtain
\begin{multline*}
 \int_{\R^N} \bigl(\varepsilon^2 \abs{\nabla (\psi u)}^2 + V \abs{\psi u}^2\bigr) 
=\int_{\R^N} g_{\varepsilon} (x, u(x))\psi(x)^2 u(x) \: dx + \int_{\R^N} \varepsilon^2 \abs{\nabla \psi}^2 \abs{u}^2  \: .
\end{multline*}
Since \(\psi = 0\) in \(\Lambda\), we deduce from \( (g_2) \) and Lemma~\ref{lemmaPositivity} that
\begin{align*}
 \int_{\R^N} g_{\varepsilon} (x, u(x))\psi (x)^2 u(x) \:dx &\leq \mu \int_{\R^N} (V+\varepsilon^2 H)\abs{\psi u}^2 \\
 &\leq \mu \int_{\R^N} \bigl(\varepsilon^2 \abs{\nabla (\psi u)}^2 + V \abs{\psi u}^2\bigr)\: .
\end{align*}
Therefore, since \(\supp \nabla \psi \subset U \setminus \Bar{\Lambda}\) and \(\inf_{U} V > 0\), we have
\[
 (1-\mu) \int_{\R^N} \bigl(\varepsilon^2 \abs{\nabla (\psi u)}^2 + V \abs{\psi u}^2\bigr)
\leq \int_{\R^N} \varepsilon^2 \abs{\nabla \psi}^2 \abs{u}^2 \leq C \varepsilon^2 \int_{U \setminus \Bar{\Lambda}} V \abs{u}^2\:.\qedhere
\]
\end{proof}

Now we have an estimate outside small balls. 
\begin{lemma}\label{lemoutsmallballs}
Let \((\varepsilon_n)_{n \in \N} \) be a sequence in \(\R^+\) such that
\(\varepsilon_n \to 0\) as \(n \to \infty\), let \((u_{n})_{n \in \N}\) be a sequence of solutions of \((\mathcal{Q}_{\varepsilon_n})\) such that
\[
 \limsup_{n \to \infty} \varepsilon_n^{-N} \mathcal{J}_{\varepsilon_n} (u_n) < \infty\: ,
\]
and let \((x^i_n)_{n \in \N} \subset \R^N\), \(1\leq i \leq M\), be sequences such that \(x^i_n \to \Bar{x}^i \in \R^N\) as \(n \to \infty\).
If for every \(i \in \{1, \dotsc, M\}\), \(V(\Bar{x}^i) > 0\) and 
\begin{equation*}
 \liminf_{n\to\infty} u_{n}(x^i_n) > 0\:,
\end{equation*}
then
\begin{equation*}
 \liminf_{R\to \infty} \liminf_{n\to \infty} \varepsilon_n^{-N}\Bigl( \int_{\R^N\setminus \mathcal{B}_n(R)} \frac{1}{2}
\bigl( \varepsilon_n^2 \abs{\nabla u_{n}}^2 + V \abs{u_{n}}^2 \bigr) - G_{\varepsilon_n}(.,u_{n}) \Bigr) \geq 0\:,
\end{equation*}
where \(\mathcal{B}_n(R) := \bigcup_{i=1}^M B(x^i_n, \varepsilon_n R)\). 
 Furthermore, if 
\begin{equation*}
 \inf\bigl\{ \limsup_{n \to \infty} \norm{u_{n}}_{L^{\infty}(U \setminus \mathcal{B}_n(R))} \st \text{\(R > 0\), \(U\) is open and \(\Bar{\Lambda} \subset U\)}\bigr\} = 0\:, 
\end{equation*}
 then
 \begin{equation*}
  \lim_{R\to \infty} \limsup_{n\to \infty} \varepsilon_n^{-N} \Bigabs{\int_{\R^N\setminus \mathcal{B}_n(R)} \frac{1}{2}
\bigl( \varepsilon_n^2 \abs{\nabla u_{n}}^2 + V \abs{u_{n}}^2 \bigr) - G_{\varepsilon_n}(.,u_{n})} = 0\:.
 \end{equation*}
\end{lemma}
The first assertion was proved in \cite{BonheureVanSchaftingen2008}*{Lemma 15}. 
\begin{proof}
First we claim that 
\[
 \lim_{R\to \infty} \limsup_{n\to \infty} \varepsilon_n^{-N} \Bigabs{\int_{\R^N\setminus \mathcal{B}_n(R)} 
\bigl( \varepsilon_n^2 \abs{\nabla u_{n}}^2 + V \abs{u_{n}}^2 -  g_{\varepsilon_n}(.,u_{n}) u_n \bigr)}= 0\: .
\]
This is proved in \cite{BonheureVanSchaftingen2008}*{Lemma 15} by taking a suitable family of test functions and using Lemma~\ref{lemuepsrescaledcompact}. We do not need to go to a subsequence since by Lemma~\ref{lemsmallballs}, 
\[
  \lim_{R \to \infty} \limsup_{n \to \infty} \int_{B(x_n^i, \varepsilon_n R) \setminus B(x_n^i, \varepsilon_n R/2)} \varepsilon_n^2 \bigl(\abs{\nabla u_{n}}^2 + V \abs{u_{n}}^2\bigr) = 0 \: .
\]

The first assertion follows, as in \cite{BonheureVanSchaftingen2008}*{Lemma 15} from the inequality
\begin{multline*}
\frac{1}{2}\int_{\R^N\setminus \mathcal{B}_n(R)} 
\bigl( \varepsilon_n^2 \abs{\nabla u_{n}}^2 + V \abs{u_{n}}^2 \bigr) - \frac{1}{2} \int_{\R^N\setminus \mathcal{B}_n(R)} g_{\varepsilon_n}(x,u_{n}(x)) u_n(x) \: dx \\
\le \frac{1}{2}\int_{\R^N\setminus \mathcal{B}_n(R)} 
\bigl( \varepsilon_n^2 \abs{\nabla u_{n}}^2 + V \abs{u_{n}}^2 \bigr) - \int_{\R^N\setminus \mathcal{B}_n(R)} G_{\varepsilon_n}(x,u_{n}(x)) \:dx\:.
\end{multline*}

For the second assertion, we have by \((g_2)\) and \((g_3)\),
\begin{multline*}
 \Bigabs{\int_{\R^N\setminus \mathcal{B}_n(R)} G_{\varepsilon_n}(x,u_{n}(x)) - \frac{1}{2} g_{\varepsilon_n}(x,u_{n}(x)) u_{n}(x)\: dx} \\
\le \int_{\R^N\setminus \mathcal{B}_n(R)} \frac{1}{2} g_{\varepsilon_n}(x,u_{n}(x)) u_{n}(x)\: dx\\
\leq \frac{1}{2} \int_{U \setminus \mathcal{B}_n(R)} (u_{n})_+^{p+1}
+\frac{\mu}{2} \int_{\R^N\setminus U} \bigl( \varepsilon_n^2 H  + V\bigr) \abs{u_{n}}^2\:.
\end{multline*}
We compute that 
\begin{equation*}
\begin{split}
\int_{U \setminus \mathcal{B}_n(R)} \abs{u_{n}}^{p+1}
&\leq C \norm{u_{n}}^{p-1}_{L^{\infty}(U \setminus \mathcal{B}_n(R))}  \int_{U \setminus \mathcal{B}_n(R)} \abs{u_{n}}^{2}   \\
&\le \frac{1}{(\inf_U V)^2}  \norm{u_{n}}^{p-1}_{L^{\infty} (U\setminus \mathcal{B}_n(R))} \int_{\R^N} \bigl( \varepsilon_n^2 \abs{\nabla u_{n}}^2 + V \abs{u_{n}}^2 \bigr)\:.
\end{split}
\end{equation*}
In view of Lemma~\ref{lemmaAsymptoticsOutsideLambda}, for every \(U \supset \Lambda\) there exists \(C > 0\) such that we have
\begin{multline*}
\Bigabs{\int_{\R^N\setminus \mathcal{B}_n(R)} G_{\varepsilon_n}(x,u_{n}(x)) - \frac{1}{2} g_{\varepsilon_n}(x,u_{n}(x)) u_{n}(x)\: dx}\\
\le \frac{1}{2} \Bigl(\frac{1}{(\inf_U V)^2}  \norm{u_{n}}^{p-1}_{L^{\infty} (U\setminus \mathcal{B}_n(R))} + C \mu \varepsilon_n^2\Bigr) \int_{\R^N} \bigl( \varepsilon_n^2 \abs{\nabla u_{n}}^2 + V \abs{u_{n}}^2 \bigr)\:.
\end{multline*}
We conclude by taking \(U \supset \Bar{\Lambda}\) small enough and \(R\) and \(n\) large enough, in view of the hypothesis and Lemma~\ref{lem:bounded}.
\end{proof}

\subsection{Conclusion}
We can now state and prove the main result of this section is

\begin{proposition}\label{estim:inf2}
Let \((\varepsilon_n)_{n \in \N} \) be a sequence in \(\R^+\) such that
\(\varepsilon_n \to 0\) as \(n \to \infty\), let \((u_{n})_{n \in \N}\) be a sequence of solutions of \(\mathcal{Q}_{\varepsilon_n}\) such that
\[
 \limsup_{n \to \infty} \varepsilon_n^{-N} \mathcal{J}_{\varepsilon_n} (u_n) < \infty\: ,
\]
and let \((x^i_n)_{n \in \N} \subset \R^N\), \(1\leq i \leq M\), be sequences such that \(x^i_n \to \Bar{x}^i \in \R^N\) as \(n \to \infty\).
If for every \(i \in \{1, \dotsc, M\}\), \(V(\Bar{x}^i) > 0\) and
\begin{equation*}
 \liminf_{n\to\infty} u_{n}(x^i_n) > 0\:,
\end{equation*}
and if for every \(i,j \in \{ 1, \dots, M \}\) such that \(i\neq j\),
\begin{equation*}
 \lim_{n\to\infty} \frac{\abs{x^i_n-x^j_n}}{\varepsilon_n} = +\infty\:,
\end{equation*}
then for every \(i \in \{ 1, \dots, M \}\), \(\Bar{x}^i \in \Bar{\Lambda}\), 
\begin{equation*}
 \lim_{n\to\infty} \frac{\dist(x^i_n, \Lambda)}{\varepsilon_n} < +\infty\:,
\end{equation*}
and
\begin{equation*}
 \liminf_{n\to\infty} \varepsilon_n^{-N} \mathcal{J}_{\varepsilon_n}(u_{\varepsilon_n}) \geq 
\sum_{i=1}^M \mathcal{C}(\Bar{x}^i)\:.
\end{equation*}
Furthermore, if 
\begin{equation*}
 \inf\Bigl\{ \limsup_{n \to \infty} \norm{u_{n}}_{L^{\infty}(U \setminus \mathcal{B}_n(R))} \st \text{\(R > 0\), \(U\) is open and \(\Bar{\Lambda} \subset U\)}\Bigr\} = 0\:, 
\end{equation*}
then
\begin{equation*}
  \lim_{n\to\infty}\varepsilon_n^{-N} \mathcal{J}_{\varepsilon_n}(u_{n})=
  \sum_{i=1}^M \mathcal{C}(\Bar{x}^i)\:.
\end{equation*}
\end{proposition}
\begin{proof}
This follows from Lemmas~\ref{lemsmallballs} and \ref{lemoutsmallballs} (see \citelist{\cite{BonheureVanSchaftingen2008}*{Proposition 16}\cite{MorozVanSchaftingen2010}*{Lemma 4.3}} for the details).
\end{proof}

\section{Asymptotics of families of almost minimizers}
\label{sectAsymptoticsMinimizers}

\subsection{Families of minimizers}

Let us recall how the results of Section~\ref{sectAsymptoticsCritical} allow to study the asymptotics of \(\inf_{\mathcal{N}_\varepsilon} \mathcal{J}_\varepsilon\).

\begin{proposition}
\label{propAsymptotics}
If \(\inf_\Lambda V > 0\), then
\[
  \lim_{\varepsilon \to 0} \varepsilon^{-N} \inf_{\mathcal{N}_\varepsilon} \mathcal{J}_\varepsilon 
= \inf_\Lambda \mathcal{C}\:.
\]
\end{proposition}

This has been proved by M.\thinspace del Pino and P.\thinspace Felmer \cite{delPinoFelmer1996}*{(2.4) and Lemma 2.2} when \(\inf_{\R^N} V > 0 \), and has been extended to decaying potentials \citelist{\cite{BonheureVanSchaftingen2008}*{Lemma 12 and proof of Proposition 21}\cite{MorozVanSchaftingen2010}*{Lemma 2.2}}.

\begin{proof}[Sketch of the proof]
First one shows that for every \(x \in \Lambda\),
\[
  \limsup_{\varepsilon \to 0} \varepsilon^{-N} \inf_{\mathcal{N}_\varepsilon} \mathcal{J}_\varepsilon 
\le \mathcal{C}(x)\:
\]
by taking suitable multiples of cutoffs of \(U_{V(x)}(\frac{\cdot-x}{\varepsilon})\).

By Proposition~\ref{prop:Minimizer}, for every \(\varepsilon > 0\), there exists \(u_\varepsilon \in \mathcal{N}_\varepsilon\) such that \(\mathcal{J}_\varepsilon(u_\varepsilon) = \inf_{\mathcal{N}_\varepsilon} \mathcal{J}_\varepsilon\). By classical regularity theory, \(u_\varepsilon\) is continuous. Choose \(x_\varepsilon \in \Bar{\Lambda}\) such that \(u_\varepsilon(x_\varepsilon) = \sup_\Lambda u_\varepsilon\). By Lemma~\ref{lem:1}, \(\liminf_{\varepsilon \to 0} u_\varepsilon(x_\varepsilon) > 0\). By Proposition~\ref{estim:inf2},
\[
 \liminf_{\varepsilon \to 0} \varepsilon^{-N} \mathcal{J}_\varepsilon(u_\varepsilon) 
\ge \inf_\Lambda \mathcal{C}\:.\qedhere
\]
\end{proof}

\subsection{Decay of almost minimizers}

The next ingredient is a decay estimate that will allow to control the functional outside \(\Lambda\) in the proof of the strict inequality \eqref{estim}. 
\begin{lemma}\label{lem:3}
Let \((\varepsilon_n)_{n \in \N} \subset \R^+_0\) be a sequence such that \(\varepsilon_n \to 0\) as \(n \to \infty\) and let \((u_n)_{n \in \N} \subset \mathcal{N}_{\varepsilon_n}\). If \(\inf_\Lambda V > 0\)  and
\begin{equation*}
 \limsup_{n\to\infty} \varepsilon_n^{-N} \mathcal{J}_{\varepsilon_n} (u_n) \leq  \inf_{\Lambda} \mathcal{C},
\end{equation*}
then, for every open set \(U \subset \R^N\) such that \(\Bar{\Lambda} \subset U\), 
\begin{equation*}
 \lim_{n\to\infty} \varepsilon_n^{-N} \int_{\R^N\setminus U} \bigl( \varepsilon_n^2 \abs{\nabla u_n}^2 + V \abs{u_n}^2 \bigr)= 0\:.
\end{equation*}
\end{lemma}

This lemma is proved by M.\thinspace del Pino and P.\thinspace Felmer \cite{delPinoFelmer1997}*{(1.19)} when \(\inf_{\R^N} V > 0\). 
The proof of Lemma~\ref{lem:3} relies on the following lemma
\begin{lemma}\label{lem:lowenergycoercivity}
Let \(U \subset \R^N\) be such that \(\Bar{\Lambda} \subset U\) and \(\inf_U V > 0\).
Let \(\psi \in C^1 (\R^N)\) and \(\varphi \in C^1 (\R^N)\) be such that \(\psi = 0\) on \(\Lambda\), \(\varphi = 0\) on \(\R^N \setminus U\), and \(\psi^2 + \varphi^2 = 1 \)  on \(\R^N\).
There exists \(C > 0\) such that for every \(\varepsilon > 0\) and \(u \in H^1_V (\R^N)\),
\[
  \mathcal{J}_{\varepsilon} ( \varphi u) + 
(1-\mu) \int_{\R^N} \bigl(\varepsilon^2 \abs{\nabla (\psi u)}^2 +  V \abs{\psi u}^2\bigr)
\le \mathcal{J}_{\varepsilon} (u) + C \varepsilon^2 \int_{U \setminus \Bar{\Lambda} } V \abs{u}^2\:.
\]
\end{lemma}
\begin{proof}
One has
\begin{multline*}
\mathcal{J}_{\varepsilon} (u) = \mathcal{J}_{\varepsilon} ( \varphi u) + \frac{1}{2}
\int_{\R^N} \bigl(\varepsilon^2 \abs{\nabla (\psi u)}^2  + V \abs{ \psi u}^2 \bigr)
-\frac{1}{2} \int_{U \setminus \Bar{\Lambda} } \varepsilon^2(\abs{\nabla \psi}^2 +\abs{\nabla \varphi}^2 ) \abs{u}^2\\
- \int_{\R^N} G_\varepsilon(x, u(x))-G_\varepsilon(x,\varphi(x) u(x))\:dx\:.
\end{multline*}
One has for every \(x \in \R^N \setminus \Lambda\), by \( (g_2) \),
 \[
 \begin{split}
   G_\varepsilon(x, u(x))-G_\varepsilon(x,\varphi(x) u(x))&= \int_{\varphi(x) u(x)}^{u(x)} g_\varepsilon(x,\sigma)\:d\sigma\\
 &\leq \mu \int_{\varphi(x) u(x)}^{u(x)} \bigl(V(x) + \varepsilon^2 H(x) \bigr) \sigma \:d\sigma\\[1.9mm]
 &= \frac{\mu}{2}\bigl(V(x) + \varepsilon^2 H(x) \bigr) \abs{\psi(x)u(x)}^2,
 \end{split}
 \]
 so that
\[
 \int_{\R^N} G_\varepsilon(x, u(x))-G_\varepsilon(x,\varphi(x) u(x))\:dx \leq \frac{\mu}{2}\int_{\R^N}\bigl(V + \varepsilon^2 H \bigr) \abs{\psi u}^2 \:.
\]
On the other hand,
\[
  \int_{U \setminus \Bar{\Lambda} } \varepsilon^2 (\abs{\nabla \psi}^2 +\abs{\nabla \varphi}^2 ) \abs{u}^2
\leq C\varepsilon^2 \int_{U \setminus \Bar{\Lambda} } V \abs{u}^2\:.
\]
The conclusion follows.
\end{proof}

\begin{proof}[Proof of Lemma~\ref{lem:3}]
Without loss of generality, assume that \(\inf_U V > 0\).
Let \(\psi \in C^1 (\R^N)\) and \(\varphi \in C^1 (\R^N)\) be such that \(\psi = 0\) on \(\Lambda\), \(\varphi = 0\) on \(\R^N \setminus U\), and \(\psi^2 + \varphi^2 = 1 \)  on \(\R^N\).
Define \(t_n\) so that \(t_n \varphi u_n \in \mathcal{N}_{\varepsilon_n}\).
By Lemma~\ref{lem:bounded},
\[
  \limsup_{n \to \infty} \varepsilon_n^{-N}\int_{\R^N} \varepsilon^2 \abs{\nabla u_n}^2 +V \abs{u_n}^2 < \infty\:,
\]
and thus
\[
  \limsup_{n \to \infty}  \varepsilon_n^{-N} \int_{\R^N} \varepsilon_n^2 \abs{\nabla (\varphi u_n)}^2 +V \abs{\varphi u_n}^2 < \infty\:.
\]
By the choice of \(t_n\),
\begin{align*}
 t_n^2\int_{\R^N} \varepsilon_n^2 \abs{\nabla (\varphi u_n)}^2 +V \abs{\varphi u_n}^2 &= \int_{\R^N} g_{\varepsilon_n}(x, t_n \varphi(x) u_n(x)) t_n \varphi(x) u_n(x)\: dx  \\
 &\geq t_n^{p+1} \int_{\Lambda}  \abs{u_n}^{p+1}\:. 
\end{align*}
We infer from Lemmas~\ref{lemENehariEstimate1}~and~\ref{lemENehariEstimate2} that
\[
  \liminf_{n \to \infty} \varepsilon_n^{-N} \int_{\Lambda}  \abs{u_n}^{p+1} > 0\:.
\]
Therefore, 
\(\limsup_{n \to \infty} t_n < \infty\)
and
\[
 \limsup_{n \to \infty}  \varepsilon_n^{-N} \int_{U \setminus \Bar{\Lambda} } V \abs{t_n u_n}^2 < \infty\:.
\]
By Lemma~\ref{lem:lowenergycoercivity},
\begin{multline*}
\limsup_{n \to \infty} \varepsilon_n^{-N}\mathcal{J}_{\varepsilon_n} (t_n u_n) \geq \liminf_{n \to \infty}  \varepsilon_n^{-N}\mathcal{J}_{\varepsilon_n} (t_n \varphi u_n) \\
+ \limsup_{n \to \infty} \varepsilon_n^{-N} (1-\mu) t_n^2 \int_{\R^N} \bigl(\varepsilon_n^2 \abs{\nabla (\psi u_n)}^2 +  V \abs{\psi u_n}^2\bigr)\:.
\end{multline*}
By assumption, we have
\[
\limsup_{n \to \infty} \varepsilon_n^{-N}\mathcal{J}_{\varepsilon_n} (t_n u_n) \le \limsup_{n \to \infty} \varepsilon_n^{-N}\mathcal{J}_{\varepsilon_n} (u_n) \le \inf_{\Lambda} \mathcal{C}\: ,
\]
and since \(t_n \varphi u_n \in \mathcal{N}_{\varepsilon_n}\), we deduce from Proposition~\ref{propAsymptotics} that 
\[
 \liminf_{n \to \infty} \varepsilon_n^{-N}\mathcal{J}_{\varepsilon_n} (t_n \varphi u_n) \ge \inf_{\Lambda} \mathcal{C}\: .
\]
Combining the last three inequalities, we obtain the conclusion.
\end{proof}

\subsection{Asymptotics of the barycenters}

As in \cite{delPinoFelmer1997}, we introduce a barycenter map in order to localize functions.
Let \(\psi \in C^1(\R^N)\) be such that \(\supp \psi\) is compact, \(\supp \psi \subset \{x \in \R^N \st V(x) > 0\}\) and \(\psi=1\) on
a neighborhood of \(\Lambda\). 
The barycenter of a function \(u \in L^2(\R^N)\) is defined by
\begin{equation*}
 \beta(u) := \frac{\displaystyle \int_{\R^N} x \abs{\psi(x) u(x)}^2 \: dx}{\displaystyle \int_{\R^N}  \abs{\psi(x)u(x)}^2\: dx}\:.
\end{equation*}
The map \(\beta\) is well-defined on the set \(\{ u \in H^1_V(\R^N) \st \psi u \ne 0\}\), which contains \(\mathcal{N}_\varepsilon\) for each \(\varepsilon > 0\) by Lemma~\ref{lemENehariEstimate2}. 

\begin{proposition}
\label{propositionAsymptoticsMoments}
Let \((\varepsilon_n)_{n\in \N} \subset \R^+_0\) be a sequence such that \(\varepsilon_n \to 0\) as \(n \to \infty\) and let \((u_n)_{n\in \N} \subset \mathcal{N}_{\varepsilon_n}\). If \(\inf_{\Lambda} V>0\) and
\begin{equation*}
 \limsup_{n \to \infty} \varepsilon_n^{-N}\mathcal{J}_{\varepsilon_n} (u_n) \leq  \inf_{\Lambda} \mathcal{C}\: ,
\end{equation*}
then
\begin{equation*}
 \lim_{n\to\infty} V\bigl(\beta(u_n)\bigr) = \inf_{\Lambda} V\:.
\end{equation*}
\end{proposition}
\begin{proof}
Let \(t_n > 0\) be such that
\begin{equation*}
 \int_{\R^N} \varepsilon_n^2 \abs{\nabla t_n \psi u_n}^2 + V_0 \abs{ t_n \psi u_n}^2
=\int_{\R^N} ( t_n \psi u_n)_+^{p+1}\:.
\end{equation*}
Define \(v_n : \R^N \to \R\) for \(y \in \R^N\) by
\[
  v_n(y) := t_n \psi(\beta(u_n)+ \varepsilon_n y)u_n(\beta(u_n)+ \varepsilon_n y)\:.
\]

\medbreak
\paragraph{Claim 1} 
\emph{The sequence \((t_n)_{n \in \N}\) is bounded.}

By Lemmas~\ref{lemENehariEstimate1} and \ref{lemENehariEstimate2},
\[
 \liminf_{n \to \infty} \epsilon_n^{- N} \int_{\R^N} (\psi u_n)_+^{p+1}
 \ge \liminf_{n \to \infty} \epsilon_n^{- N} \int_{\Lambda} (u_n)_+^{p+1}
 > 0.
\]
Since 
\[
 \limsup_{n \to \infty} \int_{\R^N} \varepsilon_n^2 \abs{\nabla t_n \psi u_n}^2 + V_0 \abs{ t_n \psi u_n}^2
\]
by Lemma~\ref{lem:bounded}, the sequence \((t_n)_{n \in \N}\) is bounded.

\medbreak
\paragraph{Claim 2} \emph{One has }
\begin{equation*}
 \limsup_{n \to \infty} \varepsilon_n^{-N} \mathcal{J}_{\varepsilon_n}(t_n \psi u_n) \leq  \inf_{\Lambda} \mathcal{C}\:.
\end{equation*}
We can write
\begin{equation*}
\begin{split}
 \mathcal{J}_{\varepsilon_n}(t_n \psi u_n)
&= \mathcal{J}_{\varepsilon_n}(t_n u_n) + \frac{\varepsilon_n^2}{2} t_n^2\int_{\R^N} \abs{\nabla (\psi u_n)}^2-\abs{\nabla u_n}^2\\
&\qquad + \frac{t_n^2}{2}  \int_{\R^N} (\psi^2-1)V \abs{u_n}^2\\
&\qquad + \int_{\R^N} G_{\varepsilon_n}(x, t_n u_n(x)) - G_{\varepsilon_n}(x, t_n \psi(x) u_n(x)) \:dx
\end{split}
\end{equation*}
Now, in view of Lemma \ref{lem:3}, since \(\psi=1\) in a neighborhood of \(\Lambda\) and \((t_n)_{n \in \N}\) is bounded,
\begin{multline*}
\frac{\varepsilon_n^2}{2}t_n^2 \int_{\R^N} \abs{\nabla (\psi u_n)}^2-\abs{\nabla u_n}^2\\
=\frac{\varepsilon_n^2}{2} t_n^2\int_{\R^N} (\psi^2-1) \abs{\nabla u_n}^2
+ 2 \psi u_n \nabla \psi \cdot \nabla u_n + \abs{\nabla \psi }^2\abs{u_n}^2
=o(\varepsilon_n^N)\: ,
\end{multline*}
as \(n \to \infty\).
Lemma \ref{lem:3} also implies that
\[
\begin{split}
 \int_{\R^N} (\psi^2-1)V \abs{u_n}^2
=o(\varepsilon_n^N)\:,
\end{split}
\]
as \(n \to \infty\).
Finally, since \(\psi = 1\) on a neighborhood \(U\) of \(\Lambda\), we deduce from \((g_2)\) that
\begin{multline*}
 \int_{\R^N} G_{\varepsilon_n}\bigl(x, t_n u_n(x)\bigr) - G_{\varepsilon_n}(x, t_n \psi(x) u_n(x))\:dx\\
\le \frac{\mu}{2} t_n^2\int_{\R^N \setminus U} \bigl(\varepsilon_n^2 H + V\bigr) \abs{u_n}^2 
=o(\varepsilon_n^N)\:,
\end{multline*}
as \(n \to \infty\).
where we have used Lemma \ref{lem:3} again.
It follows from the hypothesis that
\begin{equation*} 
  \limsup_{n \to \infty} \varepsilon_n^{-N} \mathcal{J}_{\varepsilon_n}(t_n u_n) \le \limsup_{n \to \infty} \varepsilon_n^{-N} \mathcal{J}_{\varepsilon_n}(u_n) \le \inf_{\Lambda} \mathcal{C}\:;
\end{equation*}
the claim follows.

\medbreak
\paragraph{Claim 3} \emph{There holds}
\begin{equation*}
 \limsup_{n \to \infty} \mathcal{I}_{V_0}(v_n) \leq \inf_{\Lambda} \mathcal{C}\:,
\end{equation*}
\emph{where \(V_0=\inf_\Lambda V\).}

One computes that, using \((g_2)\),
\begin{equation*}
\begin{split}
 \varepsilon_n^{N} \mathcal{I}_{V_0}(v_n)
&= \mathcal{J}_{\varepsilon_n}(t_n \psi u_n) + \frac{t_n^2}{2}  \int_{\R^N} (V_0 - V)\psi^2 \abs{u_n}^2\\
&\qquad + \int_{\R^N} G_{\varepsilon_n}(x, t_n \psi(x) u_n(x)) - \frac{t_n^{p+1}}{p+1} (\psi(x) u_n(x))_+^{p+1} \:dx\\
&\leq \mathcal{J}_{\varepsilon_n}(t_n \psi u_n) + \frac{t_n^2}{2}  \int_{\R^N} (V_0 - V)\psi^2 \abs{u_n}^2.
\end{split}
\end{equation*}
For \(\kappa \in (0, 1)\), define 
\[
  U_\kappa = \bigl\{ x \in \R^N \st \bigl(V_0-V(x)\bigr)\psi^2(x) < \kappa V(x) \bigr\}\:.
\]
Since \(V\) is continuous, \(U_\kappa\) is open and \(\Bar{\Lambda} \subset  U_\kappa\).
By Lemma \ref{lem:3}, 
\[
\begin{split}
 \int_{\R^N} (V_0 -V)\psi^2 \abs{u_n}^2
&\le \kappa\int_{U_\kappa}  V\abs{u_n}^2
+  \int_{\R^N \setminus U_\kappa} (V_0 -V)\psi^2 \abs{u_n}^2\\
&\le  \kappa \int_{\R^N} V \abs{u_n}^2 + o(\varepsilon_n^N)\:,
\end{split}
\]
as \(n \to \infty\).
Since \(\kappa > 0 \) is arbitrary, and \((t_n)_{n \in \N}\) and \( (\norm{u_n}_\varepsilon)_{n \in \N} \) are bounded,
\[
\frac{t_n^2}{2} \int_{\R^N} (V_0-V) \psi^2 \abs{u_n}^2 \le o(\varepsilon_n^N)\:,
\]
as \(n \to \infty\).
The claim now follows from Claim 2.

\medbreak
\paragraph{Conclusion} 
We know from Claim 3 that \((v_n)_{n \in \N}\) is a minimizing sequence of \(\mathcal{I}_{V_0}\) on its associated Nehari manifold 
\(
 \mathcal{M}_{V_0}.
\)
By Lemma~\ref{lemmaConvergenceMinimizingLimiting}, there exists a sequence of points \((y_n)_{n \in \N} \subset \R^N\) such that \(v_n(\cdot-y_n)\) converges in \(H^1(\R^N)\) to the positive solution
\(U_{V_0}\) of problem \eqref{problemLimit}.
Let \(x_n := \varepsilon_n y_n\). 
Since
\[
\begin{split}
 \beta(u_n)&=x_n + \frac{\displaystyle \int_{\R^N} (x-x_n) \abs{\psi(x) u_n(x)}^2 \: dx}{\displaystyle \int_{\R^N}  \abs{\psi(x) u_n(x)}^2\: dx}\\
&=x_n + \frac{\displaystyle \varepsilon_n \int_{\R^N} y \abs{v_n(y-y_n)}^2 \: dy}{\displaystyle \int_{\R^N}  \abs{v_n(y-y_n)}^2 \: dy}\:,
\end{split}
\]
we have \(\lim_{n \to \infty} \beta(u_n)-x_n = 0\).

Now, note that for \(n\) large enough, by Lemma~\ref{lemENehariEstimate2}, 
\[ 
 \liminf_{n\to \infty} \varepsilon_n^{-N}\int_{\Lambda} \abs{u_n}^2 >0.
\]
Since \(v_n(\cdot-y_n) \to U_{V_0}\) in \(L^2(\R^N)\), we must have \(\dist(x_n, \Lambda) = O(\varepsilon_n)\) as \(n \to \infty\).
Let
\[
  \Bar{V} := \limsup_{n \to \infty} V(x_n)= \limsup_{n \to \infty} V(\beta(u_n)).
\]
Since \(V\) is continuous on the compact set \(\supp \psi\), one has \(\lim_{n \to \infty} V(x_n+\varepsilon y) \geq V_0 \).

By Claim 2,
\[
\begin{split}
 b_{V_0} &\geq \varepsilon_n^{-N} \mathcal{J}_{\varepsilon_n}(t_n \psi u_n) + o(1)\\
&\geq \frac{1}{2}\int_{\R^N} \abs{\nabla v_n}^2 + \frac{1}{2}\int_{\R^N} V(x_n+\varepsilon y) \abs{v_n}^2
- \frac{1}{p+1}\int_{\R^N} (v_n)^{p+1}_+ + o(1),
\end{split}
\]
as \(n \to \infty\),
and thus we obtain
\begin{multline*}
\frac{1}{2}\int_{\R^N} \abs{\nabla U_{V_0}}^2 + V_0 \abs{U_{V_0}}^2
- \frac{1}{p+1}\int_{\R^N} (U_{V_0})^{p+1}_+ = b_{V_0} \\
\geq \frac{1}{2}\int_{\R^N} \bigl(\abs{\nabla U_{V_0}}^2 + \Bar{V} \abs{U_{V_0}}^2\bigr)
- \frac{1}{p+1}\int_{\R^N} (U_{V_0})^{p+1}_+.
\end{multline*}
This implies that \(\Bar{V} \leq V_0\). The conclusion follows.
\end{proof}

\section{The minimax level}
\label{sectionMinimax}

\subsection{Definition of the minimax level}

Following M.\thinspace del Pino and P.\thinspace Felmer \cites{delPinoFelmer1997,delPinoFelmer2002}, we define a minimax
value for \(\mathcal{J}_{\varepsilon}\). 

Let \(\eta \in C^{\infty}_{\mathrm{c}}(\R^+)\) be a cut-off function such that \(0 \leq \eta \leq 1\), \(\eta=1\) on a neighborhood of \(\Lambda\) and \(\supp \eta \subset \{x \in \R^N : V(x) > 0\}\).
We define \(w_{\varepsilon,y} \in C^\infty_{\mathrm{c}}(\R^N)\) by 
\begin{equation}\label{wepsy}
 w_{\varepsilon,y}(x) := t_{\varepsilon,y} \eta(x)
U_{V(y)}\Bigl(\frac{y-x}{\varepsilon}\Bigr),
\end{equation}
where \(t_{\varepsilon,y} > 0\) is such that \(w_{\varepsilon,y} \in \mathcal{N}_{\varepsilon}\).
Let \(\Lambda_\varepsilon \subset \Lambda\) be such that
\begin{equation}
\label{limsupLambdaepsilon}
  \lim_{\varepsilon \to 0} \sup_{x \in \partial \Lambda_\varepsilon} \dist(x,\partial\Lambda) = 0
\end{equation}
and
\begin{equation}
\label{liminfLambdaepsilon}
 \lim_{\varepsilon \to 0} \inf_{x \in \partial \Lambda_\varepsilon} \frac{\dist(x,\partial\Lambda)}{\varepsilon} = \infty\:.
\end{equation}

We define the family of paths
\begin{equation*}
 \Gamma_{\varepsilon} := \left\lbrace \gamma \in C(\overline{\Lambda}_\varepsilon, \mathcal{N}_{\varepsilon}) \st \text{for every \(y \in \partial\Lambda_\varepsilon\), } \gamma(y) = w_{\varepsilon,y} \right\rbrace
\end{equation*}
and the minimax value
\begin{equation}\label{ceps}
 c_{\varepsilon} := \inf_{\gamma \in \Gamma_{\varepsilon}} \sup_{y \in  \Lambda_\varepsilon} \mathcal{J}_{\varepsilon}( \gamma(y)
)\:.
\end{equation}

We want to apply the following theorem.
\begin{theorem}[General Minimax Principle \cite{Willem1996}*{Theorem 2.9}]
 Let \(X\) be a Banach space. Let \(M_0\) be a closed subspace of the metric space \(M\) and $\Gamma_0 \subset
C(M_0,X)$. Define
\begin{equation*}
 \Gamma := \left\lbrace \gamma \in C(M,X) \st \gamma_{\vert M_0} \in \Gamma_0 \right\rbrace.
\end{equation*}
If \(\varphi \in C^1(X,\R)\) satisfies
\begin{equation*}
 \infty > c := \inf_{\gamma \in \Gamma} \sup_{z \in M} \varphi\bigl(\gamma(z)\bigr) > a := \sup_{\gamma_0 \in \Gamma_0}
 \sup_{z \in M_0} \varphi\bigl(\gamma_0(z)\bigr)
\end{equation*}
and if \(\varphi\) satisfies the Palais-Smale condition at the level \(c\), then \(c\) is a critical value of \(\varphi\).
\end{theorem}

Since \(\mathcal{J}_{\varepsilon}\) satisfies the Palais-Smale condition (Lemma~\ref{lem:PalaisSmale}), we have
to show that for \(\varepsilon > 0\) small enough,
\begin{equation}\label{estim}
 c_{\varepsilon} > \sup_{y \in \partial\Lambda_\varepsilon} \mathcal{J}_{\varepsilon}( w_{\varepsilon,y} ) =: a_{\varepsilon}\:.
\end{equation}

\subsection{Estimates on the levels}\label{sectionEstimatesLevels}

\subsubsection{Estimate of \(a_{\varepsilon}\)}
We begin with an estimate of \(a_{\varepsilon}\). 
\begin{lemma}\label{lem:estimaeps}
We have
\begin{equation*}
 \lim_{\varepsilon \to 0}\varepsilon^{-N} a_{\varepsilon} = \sup_{\partial\Lambda} \mathcal{C}\:,
\end{equation*}
\end{lemma}
\begin{proof}
 By a standard computation, we find in view of \eqref{liminfLambdaepsilon}
\begin{equation*}
 \mathcal{J}_{\varepsilon}( w_{\varepsilon,y} ) = \varepsilon^N \mathcal{I}_{V(y)}\left(U_{V(y)}\right) + o(\varepsilon^N)\:.
\end{equation*}
as \(\varepsilon \to 0\), uniformly in \(y \in \Lambda\).
Thus
\begin{equation*}
a_{\varepsilon} = \sup_{y \in \partial\Lambda_\varepsilon}  \mathcal{J}_{\varepsilon}( w_{\varepsilon,y} ) = \varepsilon^N \sup_{y \in \partial\Lambda_\varepsilon} 
b_{V(y)} + o(\varepsilon^N)\:.
\end{equation*}
The estimate follows from \eqref{estim:by}, \eqref{limsupLambdaepsilon} and the continuity of \(V\).
\end{proof}

\subsubsection{Upper estimate of the critical level \(c_{\varepsilon}\)}
The same method gives an upper estimate on \(c_{\varepsilon}\). 

\begin{lemma}\label{estim:sup2}
We have
\begin{equation*}
 \limsup_{\varepsilon \to 0}\varepsilon^{-N} c_{\varepsilon} \le \sup_{\Lambda} \mathcal{C}\:.
\end{equation*}
\end{lemma}
\begin{proof}
As a test path in \eqref{ceps}, we take \(w_{\varepsilon,y}\) defined by \eqref{wepsy} for every \(y \in \Lambda_\varepsilon\). We obtain the first
estimate after a straightforward computation in view of \eqref{liminfLambdaepsilon}. 
\end{proof}

\subsubsection{Lower estimate of the critical level \(c_{\varepsilon}\)}
A more delicate construction gives a lower estimate of the critical level \(c_{\varepsilon}\).

\begin{lemma}
\label{lemmaEstimatecepsilon}
If 
\[
 \sup_{\Lambda} \mathcal{C} > \inf_{\Lambda} \mathcal{C},
\]
then
\[ 
 \liminf_{\varepsilon \to 0} \varepsilon^{-N} c_{\varepsilon} > \inf_{\Lambda} \mathcal{C}.
\]
\end{lemma}

We do not know whether one has the natural stronger conclusion 
\[
 \liminf_{\varepsilon \to 0} \varepsilon^{-N} c_{\varepsilon} \geq \sup_{\Lambda} \mathcal{C}.
\]

\begin{lemma}\label{lem:barycenter}
Let \( x \in \Lambda\). There exists \(\varepsilon_0 > 0\) such that for every \(\varepsilon \in{} (0, \varepsilon_0)\) and every \(\gamma \in \Gamma_{\varepsilon}\), there exists \(z \in \Lambda_\varepsilon\) such that \(\beta (\gamma(z)) = x\).
\end{lemma}

\begin{proof}
For every \(z \in \partial \Lambda_\varepsilon\), one has in view of the definition of \(w_{\varepsilon, y}\),
\[
 \beta(w_{\varepsilon, y} )=y + o(1), 
\]
as \(\varepsilon \to 0\), uniformly in \(z\).

Let \(\gamma \in \Gamma_\varepsilon\).
Therefore, if \(\varepsilon\) is small enough, one has for every 
\(y \in \partial \Lambda_\varepsilon\), by \eqref{limsupLambdaepsilon}, \(x \in \Lambda_{\varepsilon}\) for \(n\) large enough.

When \(\varepsilon\) is small enough, we have 
\[
 \sup_{y \in \partial \Lambda_{\varepsilon}} \abs{\beta(w_{\varepsilon, y} )-y} < \inf_{y \in \partial \Lambda_{\varepsilon}} \abs{y-x},
\]
therefore, by the properties of the topological degree,  if \(x \in \Lambda_{\varepsilon}\) there exists  \(z \in \Lambda_\varepsilon\) such that \(\beta (\gamma(z)) = x\).
\end{proof}

We follow the arguments of \cite{delPinoFelmer1997}. Heuristically, the idea is to show that a sequence of functions violating the strict 
inequality \eqref{estim} cannot have enough energy to stay concentrated inside \(\Lambda\) and must thus concentrate around a
point of \(\partial \Lambda\). But this would in fact contradict the continuity of the paths in \(\Gamma_{\varepsilon}\). 

\begin{proof}[Proof of Lemma~\ref{lemmaEstimatecepsilon}]
Assume by contradiction that there is a sequence \( (\varepsilon_n)_{n \in \N}\) such that \(\lim_{n \to \infty} \varepsilon_n = 0\) and \(\lim_{n \to \infty} \varepsilon_n^{-N} c_{\varepsilon_n} \le \inf_{\Lambda} \mathcal{C}\). 
By definition of \(c_\varepsilon\), there exists \(\gamma_n \in \Gamma_{\varepsilon_n}\) such that
\[
 \lim_{n \to \infty} \sup_{x \in \Lambda_{\varepsilon_n}} \varepsilon_n^{-N}\mathcal{J}_{\varepsilon_n}\bigl(\gamma_n(x)\bigr) \le \inf_{\Lambda} \mathcal{C}\:.
\]
Choose \(x\in \Lambda\) such that \(V(x) > \inf_{\Lambda} V\).
For each \(n \in \N\) large enough, let \(x_n\) be given by Lemma \ref{lem:barycenter} so that \(\beta(\gamma_n(x_n)) = x\). One has
\[
 \limsup_{n \to \infty} \varepsilon_n^{-N}\mathcal{J}_{\varepsilon_n}\bigl(\gamma_n(x_n)\bigr) \leq \inf_{\Lambda} \mathcal{C}\:.
\]
Proposition~\ref{propositionAsymptoticsMoments} brings then a contradiction.
\end{proof}

\subsection{Existence of a solution}

We are now in a position to prove the strict inequality \eqref{estim}.

\begin{lemma}
\label{lemmaStrict}
If 
\[
 \sup_{\Lambda} \mathcal{C} > \inf_{\Lambda} \mathcal{C}= \sup_{\partial\Lambda} \mathcal{C}\:,
\]
then
\[ 
 \liminf_{\varepsilon \to 0} \varepsilon^{-N} (c_{\varepsilon} - a_{\varepsilon})>0\:.
\]
\end{lemma}
\begin{proof}
This follows directly from Lemmas~\ref{lem:estimaeps} and \ref{lemmaEstimatecepsilon}.
\end{proof}

As a consequence of the General Minimax Principle, we have thus proved the following existence result for the penalized problem \eqref{problemPNLSE}.
\begin{proposition}
 Let \(N\geq 3\), \(1 < p < \frac{N+2}{N-2}\) and let \(g_{\varepsilon}: \R^N \times \R^+ \rightarrow \R\) be a function satisfying assumptions \((g_1)\)-
\((g_4)\). For \(\varepsilon > 0\) small enough, there exists \(u_\varepsilon \in \mathcal{N}_\varepsilon\) such that \(\mathcal{J}_{\varepsilon}(u_\varepsilon)= c_{\varepsilon}\) and \(\mathcal{J}_{\varepsilon}'(u_\varepsilon)= 0\). 
\end{proposition}
\begin{proof}
This follows from the general minimax principle (Theorem~\ref{lem:estimaeps}), the Palais-Smale condition coming from Lemma~\ref{lem:PalaisSmale} and the strict inequality of Lemma~\ref{lemmaStrict}.
\end{proof}

\section{Back to the original problem}

\label{sectionOriginalProblem}

\subsection{Asymptotics of solutions}
Thanks to the asymptotics of solutions of Section~\ref{sectAsymptoticsCritical} and the estimates on the critical level of Section~\ref{sectionEstimatesLevels}, we prove that the solution \(u_{\varepsilon}\) is single-peaked.

\begin{lemma}\label{lem:unifdecay}
Let \((u_{\varepsilon})_{\varepsilon>0}\) be a family such that for \(\varepsilon > 0\) small enough, \(\mathcal{J}_{\varepsilon}(u_\varepsilon)= c_{\varepsilon}\) and \(\mathcal{J}_{\varepsilon}'(u_\varepsilon)= 0\).
Let \((x_{\varepsilon})_{\varepsilon>0}\) in \(\Lambda\) be such that
\[
 \liminf_{\varepsilon\to 0} u_{\varepsilon}(x_{\varepsilon}) > 0\:.
\]
If
\[
\sup_{\Lambda}  \mathcal{C} < 2 \inf_{\Lambda}  \mathcal{C}\: ,
\]
then for every \(U \subset \R^N\) such that \(\Bar{U}\) is compact and \(\inf_U V > 0\),
\[
 \lim_{\substack{\varepsilon\to 0 \\ R \to \infty}} \norm{u_{\varepsilon}}_{L^{\infty}(U \setminus
 B(x_{\varepsilon},\varepsilon R))}   = 0\:.
\]
If moreover
\[
  \sup_{\Lambda} \mathcal{C} > \inf_{\Lambda} \mathcal{C}= \sup_{\partial\Lambda} \mathcal{C}\:,
\]
then
\[
 \liminf_{\varepsilon\to 0} d(x_{\varepsilon}, \R^N \setminus \Lambda) > 0\:.
\]
\end{lemma}
\begin{proof}
First we prove that
\[
 \lim_{\substack{\varepsilon\to 0 \\ R \to \infty}} \norm{u_{\varepsilon}}_{L^{\infty}(U \setminus
 B(x_{\varepsilon},\varepsilon R))} = 0\:.
\]
Assume by contradiction that there exist sequences \((\varepsilon_n)_{n\in \N}\) and
\((y_n)_{n\in \N}\) such that for every \(n \in \N\), \(y_n \in U\),
\begin{align*}
 \lim_{n\to \infty} \varepsilon_n &= 0\:, & \liminf_{n\to \infty} u_{\varepsilon_n}(y_n) &> 0 &
&\text{and} & \lim_{n\to \infty} \frac{\abs{x_{\varepsilon_n}-y_n}}{\varepsilon_n} &= +\infty\:.
\end{align*}
Then, by Lemma~\ref{estim:sup2},
\begin{equation*}
 \limsup_{n\to\infty} \varepsilon_n^{-N} \mathcal{J}_{\varepsilon_n}(u_{\varepsilon_n}) \le \sup_{\Lambda} \mathcal{C}\:,
\end{equation*} 
while by Proposition~\ref{estim:inf2}
\begin{equation*}
  \liminf_{n\to\infty} \varepsilon_n^{-N} \mathcal{J}_{\varepsilon_n}(u_{\varepsilon_n}) 
   \geq \liminf_{n\to\infty} \bigl( \mathcal{C}(x_{\varepsilon_n}) 
    + \mathcal{C}(y_n) \bigr) \geq 2 \inf_{\Lambda} \mathcal{C}\: .
\end{equation*}
This is a contradiction with our assumption.

\medbreak
Now we turn to the second assertion. Assume by contradiction that there exists a sequence \((\varepsilon_n)_{n\in \N}\) such that
\(\varepsilon_n \to 0\) and \(\lim_{n\to \infty} \dist (x_{\varepsilon_n}, \R^N\setminus \Lambda) = 0\). Then, by the first assertion, the second part of Proposition~\ref{estim:inf2}
and Lemma~\ref{lemmaEstimatecepsilon},
\begin{equation*}
\inf_{\Lambda} \mathcal{C}
< \liminf_{n\to\infty}  \varepsilon_n^{-N} \mathcal{J}_{\varepsilon_n}(u_{\varepsilon_n}) 
= \liminf_{n\to\infty} \mathcal{C}(x_{\varepsilon_n})
= \sup_{\partial\Lambda} \mathcal{C}\:.
\end{equation*}
But this contradicts our assumption.
\end{proof}

This allows now to show that \(u_\varepsilon\) is a subsolution for some second-order linear elliptic operator.

\begin{lemma}
\label{lemmaIneqSmallBalls}
 Let \((u_{\varepsilon})_{\varepsilon>0}\) be a family of positive solutions of \eqref{problemPNLSE} at the level \(c_{\varepsilon}\)
and let \((x_{\varepsilon})_{\varepsilon>0} \subset \Lambda\) be such that
\[
 \liminf_{\varepsilon\to 0} u_{\varepsilon}(x_{\varepsilon}) > 0.
\]
Then there exists \(\varepsilon_0 > 0\) and \(R>0\) such that for all \(\varepsilon \in (0,\varepsilon_0)\),
\[
 -\varepsilon^2 (\Delta + \mu H) u_{\varepsilon} + (1-\mu) V u_{\varepsilon} \leq 0 \hspace{0.5cm} \text{in} \ \R^N\setminus
 B(x_{\varepsilon},\varepsilon R).
\]
\end{lemma}
\begin{proof}
This follows from Lemma~\ref{lem:unifdecay}, see \cite{MorozVanSchaftingen2010}*{Lemma 5.1} for the details.
\end{proof}

We then have a comparison principle \cite{MorozVanSchaftingen2010}*{Lemma 3.2}.

\begin{lemma}
\label{lemmaComparison}       
Let \(\Omega \subset \R^N\) be a domain with a smooth boundary. 
Let $v \in H^1_V (\Omega)$ and \(w \in H^1_\mathrm{loc} (\Bar{\Omega})\) be such that \(w \ge 0\) in \(\Omega\).
If
\begin{equation*}
 - \varepsilon^2 (\Delta + \mu H) v + (1 - \mu) V v
\le - \varepsilon^2 (\Delta + \mu H) w + (1 - \mu) V w, \qquad \text{weakly in \(\Omega\)}.
\end{equation*}
and \(v \leq w\) on \(\partial\Omega\),
then \(v \leq w\) in \(\Omega\).
\end{lemma}
\begin{proof}
Take \(\psi \in C^\infty_c(\R^N)\) such that \(\psi \equiv 1\) on \(B_1\) and \(\supp \psi \subset B_2\) and define \(\psi_n\in C^\infty_c(\R^N)\)
for \(n \in \N\) and \(x \in \R^N\) by \(\psi_n(x)=\psi(x/n)\). 
Using \( \psi_n^2(v - w)_+ \in H^1_V(\Omega)\) as a test function in the inequation, one has
\begin{multline*}
 \int_{\Omega} \bigl(\varepsilon^2 \abs{\nabla (\psi_n (v - w)_+)}^2 - \varepsilon^2 \mu H \abs{\psi_n (v - w)_+}^2 + (1 - \mu) V \abs{\psi_n (v -w)_+}^2 \bigr)\\
 \le \varepsilon^2 \int_{\Omega} \abs{\nabla \psi_n}^2 (v - w)_+^2\:.
\end{multline*}
By Lemma~\ref{lemmaPositivity} on the one hand and by  
definition of \(\psi_n\) and nonnegativity of \(w\) on the other hand, we have
\[
(1 - \mu) \int_{\Omega} \varepsilon^2 \abs{\nabla (\psi_n (v - w)_+)}^2 +  V \abs{\psi_n (v - w)_+}^2
\le C \int_{B_{2 n} \setminus B_n} \frac{\abs{v(x)}^2}{\abs{x}^2}\,dx.
\]
By Lebesgue dominated convergence, we deduce that \( \psi_n (v - w)_+ \to 0\) strongly in \(H^1_V (\Omega)\) as \(n \to \infty\). Hence, \( (v - w)_+ = 0\).
\end{proof}

\subsection{Barrier functions and solution of the original problem}
Since \(u_\varepsilon\) is a subsolution for some second-order linear elliptic operator, we shall compare it with supersolutions of that operator. We first recall how suitable supersolutions can be constructed.

\subsubsection{The case of fast decaying potentials}

Indepentently of the decay rate of \(V\) we have

\begin{lemma}
\label{lemmaBarrierFast}
Let \((x_{\varepsilon})_{\varepsilon} \subset \Lambda\) be such that \(\liminf_{\varepsilon\to 0} d(x_{\varepsilon},\partial \Lambda) > 0\), let \(\mu \in (0,1)\) and let \(R>0\).
If \(N \ge 3\), then there exists \(\varepsilon_0 > 0\) and a family of functions \((W_{\varepsilon})_{0<\varepsilon<\varepsilon_0}\) in \(C^{1,1}(\R^N\setminus B(x_{\varepsilon},\varepsilon R))\)
such that, for \(\varepsilon \in (0,\varepsilon_0)\),
\begin{enumerate}[(i)]
\item \(W_{\varepsilon}\) satisfies the inequation
\[
 -\varepsilon^2 (\Delta + \mu H) W_{\varepsilon} + (1-\mu) V W_{\varepsilon} \geq 0 \hspace{0.5cm} \text{in} \ \R^N\setminus
 B(x_{\varepsilon},\varepsilon R),
\]
\item \(\nabla W_{\varepsilon} \in L^2(\R^N\setminus B(x_{\varepsilon},\varepsilon R))\),
\item \(W_{\varepsilon} = 1\) on \(\partial B(x_{\varepsilon},\varepsilon R)\),
\item there exist \(C, \lambda, \nu > 0\) such that for every \(x \in \R^N \setminus B(x_{\varepsilon},\varepsilon R)\),
\[
 W_{\varepsilon}(x) \leq C \exp \left( -\frac{\lambda}{\varepsilon} \frac{\abs{x-x_{\varepsilon}}}{1+\abs{x-x_{\varepsilon}}} \right) \left( 1+\abs{x}^2 \right)^{-\frac{N-2}{2}}.
\]
\end{enumerate}
\end{lemma}
\begin{proof}
The arguments are the same as those of V.\thinspace Moroz and J.\thinspace Van Schaftingen \cite{MorozVanSchaftingen2010}*{Lemma 5.2}, since the penalization potential \(H\) is the same.
\end{proof}

The decay of \(u_\varepsilon\) is then similar to the decay of \(W_\varepsilon\).

\begin{proposition}
\label{propositionDecayFast}
Let \((u_{\varepsilon})_{\varepsilon>0}\) be a family of positive solutions of \eqref{problemPNLSE} at the level \(c_{\varepsilon}\)
and let \((x_{\varepsilon})_{\varepsilon>0} \subset \Lambda\) be such that
\[
 \liminf_{\varepsilon\to 0} u_{\varepsilon}(x_{\varepsilon}) > 0\:.
\]
If 
\(N \ge 3\) 
then there exists \(C,\lambda > 0\) and \(\varepsilon_0 > 0\) and \(R>0\) such that for all \(\varepsilon \in (0,\varepsilon_0)\),
\begin{equation*}
  u_{\varepsilon}(x) \leq C \exp \left( -\frac{\lambda}{\varepsilon} \frac{\abs{x-x_{\varepsilon}}}{1+\abs{x-x_{\varepsilon}}} \right) 
\left( 1+\abs{x}^2 \right)^{-\frac{N-2}{2}},
 \qquad x \in \R^N.
\end{equation*}
\end{proposition}

\begin{proof}
This is a consequence of Lemmas~\ref{lemmaIneqSmallBalls} and \ref{lemmaBarrierFast} together with the comparison principle (Lemma~\ref{lemmaComparison}).
\end{proof}

We can now go back to the original problem.

\begin{proposition}
\label{propOriginalFast}
Let \((u_{\varepsilon})_{\varepsilon>0}\) be a family of positive solutions of \eqref{problemPNLSE} at the level \(c_{\varepsilon}\). 
If
\(
 \frac{1}{p} < \frac{N-2}{N},
\)
then there exists \(\varepsilon_0 > 0\) such that, for every \(0 < \varepsilon < \varepsilon_0\), \(u_{\varepsilon}\) solves the original
problem \eqref{problemNLSE}.
\end{proposition}
\begin{proof}
The proof follows the lines of \cite{MorozVanSchaftingen2010}*{Proposition 5.4}. By Lemma \ref{lem:1},
there exists a family of points \((x_{\varepsilon})_{\varepsilon>0} \subset \Lambda\) such that
\[
 \liminf_{\varepsilon\to 0} u_{\varepsilon}(x_{\varepsilon}) > 0.
\]
By Lemma \ref{lem:unifdecay}, \(d_0 := \inf d(x_{\varepsilon}, \partial \Lambda) > 0\). Hence, by
proposition~\ref{propositionDecayFast}, we have for \(\varepsilon > 0\) small enough and for \(x \in \R^N\setminus\Lambda\),
\begin{equation*}
\begin{split}
 \left(u_{\varepsilon}(x)\right)^{p-1} &\leq \Bigl( C \exp \Bigl( -\frac{\lambda}{\varepsilon} \frac{\abs{x-x_{\varepsilon}}}{1+\abs{x-x_{\varepsilon}}} \Bigr) 
\bigl( 1+\abs{x}^2 \bigr)^{-\frac{N-2}{2}} \Bigr)^{p-1} \\
&\leq \mu \varepsilon^2 \frac{(N-2)^2}{4\abs{x-x_0}^2} \Bigl(\frac{\log \frac{\rho}{\rho_0}}{\log \frac{\abs{x-x_0}}{\rho_0}}
\Bigr)^{1+\beta}
= \mu \varepsilon^2 H(x).
\end{split}
\end{equation*}
By definition of the penalized nonlinearity \(g_{\varepsilon}\), one has then 
\begin{equation*}
g_{\varepsilon}\left( x, u_{\varepsilon}(x) \right) = \left( u_{\varepsilon}(x) \right)^p, \qquad x \in \R^N\setminus\Lambda,
\end{equation*}
and therefore \(u_{\varepsilon}\) solves the original problem \eqref{problemNLSE}.
\end{proof}

\subsubsection{The case of slow decaying potentials}

Now we assume that \(\liminf_{\abs{x} \to \infty} V(x) \abs{x}^2\). We first have a counterpart of Lemma~\ref{lemmaBarrierFast}.

\begin{lemma}
Let \((x_{\varepsilon})_{\varepsilon} \subset \Lambda\) be such that \(\liminf_{\varepsilon\to 0} d(x_{\varepsilon},\partial \Lambda) > 0\), let \(\mu \in (0,1)\) and let \(R>0\).
If 
\[
  \liminf_{\abs{x} \to \infty} V(x) \abs{x}^2 > 0,
\]
then, there exists \(\varepsilon_0 > 0\) and a family of functions \((W_{\varepsilon})_{0<\varepsilon<\varepsilon_0}\) in \(C^{1,1}(\R^N\setminus B(x_{\varepsilon},\varepsilon R))\)
such that, for \(\varepsilon \in (0,\varepsilon_0)\),
\begin{enumerate}[(i)]
\item \(W_{\varepsilon}\) satisfies the inequation
\[
 -\varepsilon^2 (\Delta + \mu H) W_{\varepsilon} + (1-\mu) V W_{\varepsilon} \geq 0 \hspace{0.5cm} \text{in} \ \R^N\setminus
 B(x_{\varepsilon},\varepsilon R),
\]
\item \(\nabla W_{\varepsilon} \in L^2(\R^N\setminus B(x_{\varepsilon},\varepsilon R))\),
\item \(W_{\varepsilon} = 1\) on \(\partial B(x_{\varepsilon},\varepsilon R)\),
\item there exist \(C, \lambda, \nu > 0\) such that for every \(x \in \R^N\setminus B(x_{\varepsilon},\varepsilon R)\),
\[
 W_{\varepsilon}(x) \leq C \exp \Bigl( -\frac{\lambda}{\varepsilon} \frac{\abs{x-x_{\varepsilon}}}{1+\abs{x-x_{\varepsilon}}} \Bigr) \bigl( 1+\abs{x}^2 \bigr)^{-\frac{\nu}{\epsilon}}.
\]
\end{enumerate}
\end{lemma}
\begin{proof}
See the discussion after \cite{MorozVanSchaftingen2010}*{Theorem 7}.
\end{proof}

As a consequence we have
\begin{proposition}
\label{propositionDecaySlow}
Let \((u_{\varepsilon})_{\varepsilon>0}\) be a family of positive solutions of \eqref{problemPNLSE} at the level \(c_{\varepsilon}\)
and let \((x_{\varepsilon})_{\varepsilon>0} \subset \Lambda\) be such that
\[
 \liminf_{\varepsilon\to 0} u_{\varepsilon}(x_{\varepsilon}) > 0\:.
\]
If 
\[
  \liminf_{\abs{x} \to \infty} V(x) \abs{x}^2 > 0,
\] 
then there exists \(C,\lambda, \nu > 0\) and \(\varepsilon_0 > 0\) and \(R>0\) such that for all \(\varepsilon \in (0,\varepsilon_0)\),
\begin{equation*}
  u_{\varepsilon}(x) \leq C \exp \left( -\frac{\lambda}{\varepsilon} \frac{\abs{x-x_{\varepsilon}}}{1+\abs{x-x_{\varepsilon}}} \right) 
\bigl( 1+\abs{x}^2 \bigr)^{-\frac{\nu}{\epsilon}},
 \qquad x \in \R^N.
\end{equation*}
\end{proposition}

This allows us to go back to our original problem.

\begin{proposition}
\label{propOriginalSlow}
 Let \((u_{\varepsilon})_{\varepsilon>0}\) be a family of positive solutions of \eqref{problemPNLSE} at the level \(c_{\varepsilon}\). 
If
\begin{equation*}
 \liminf_{\abs{x} \to \infty} V(x) \abs{x}^2 > 0,
\end{equation*}
then there exists \(\varepsilon_0 > 0\) such that, for every \(0 < \varepsilon < \varepsilon_0\), \(u_{\varepsilon}\) solves the original
problem \eqref{problemNLSE}.
\end{proposition}
\begin{proof}
The proof begins as the proof of Proposition~\ref{propOriginalFast}. Applying
proposition~\ref{propositionDecaySlow}, we have for \(\varepsilon > 0\) small enough and for \(x \in \R^N\setminus\Lambda\),
\begin{equation*}
\begin{split}
 \bigl(u_{\varepsilon}(x)\bigr)^{p-1} &\leq \Bigl( C \exp \Bigl( -\frac{\lambda}{\varepsilon} \frac{\abs{x-x_{\varepsilon}}}{1+\abs{x-x_{\varepsilon}}} \Bigr) 
\bigl( 1+\abs{x}^2 \bigr)^{-\frac{\nu}{\epsilon}} \Bigr)^{p-1} \\
&\leq \mu \varepsilon^2 \frac{(N-2)^2}{4\abs{x-x_0}^2 \bigl(\log \frac{\abs{x-x_0}}{\rho_0}
\bigr)^{1+\beta}}
= \mu \varepsilon^2 H(x)\:,
\end{split}
\end{equation*}
and therefore \(u_{\varepsilon}\) solves the original problem \eqref{problemNLSE}.
\end{proof}

\subsubsection{Proof of the main theorem}

Finally we complete the proof of the main theorem.

\begin{proof}[Proof of Theorem \ref{theoremMainLambda}] 
Let \(u_\varepsilon\) be the solution of the penalized problem \eqref{problemPNLSE} from Proposition~\ref{prop:Minimizer}. By Lemma \ref{lem:1}, there exists \((x_\varepsilon)_{\varepsilon > 0}\) such that 
\[
 \liminf_{\varepsilon \to 0} u_\varepsilon(x_\varepsilon) > 0.
\]
By Lemma \ref{lem:unifdecay},
\[
  \liminf_{\varepsilon\to 0} d(x_{\varepsilon}, \R^N \setminus \Lambda) > 0\:.
\]
By Proposition~\ref{propOriginalFast} or \ref{propOriginalSlow},
\(u_\varepsilon\) solves \eqref{problemNLSE} for \(\varepsilon\) small enough.
\end{proof}

\section{The low-dimensional case}
\label{sectionSlow}
In the case \(N \le 2\), we do not have the Hardy inequality, but since one cannot have \(\frac{1}{p} < \frac{N-2}{N}\), we can use some information about the decay of \(V\).

We define the penalization
potential \(H : \R^N \to \R\) by
\begin{equation*}
 H(x) := \bigl(1-\chi_{\Lambda}(x)\bigr) \frac{1}{\abs{x-x_0}^{2 + \beta}}
\end{equation*}
for some \(\beta > 0\).

In place of Lemma~\ref{lemmaPositivity}, we now have

\begin{lemma}
For every \(u \in C^{\infty}_c(\R^N)\), 
\begin{equation*}
 \int_{\R^N} \bigl( \abs{\nabla u}^2 - H \abs{u}^2   \bigr) \geq - C \int_{\R^N} V \abs{u}^2\:.
\end{equation*}
\end{lemma}
\begin{proof}
Since \(\liminf_{\abs{x} \to \infty} V(x) \abs{x}^2 > 0\), there exists \(R > 0\) such that if \(x \in R^N \setminus B_R\), \(H(x) \ge V(x)\). One has 
\[
  \int_{\R^N \setminus B_R} H \abs{u}^2 \leq \int_{\R^N \setminus B_R} V \abs{u}^2\:.
\]
Taking \(\varphi \in C^\infty (\R^N)\) such that \(\varphi \ge 0\), \(\supp \varphi \subset B_{2R}\) and \(\varphi = 1\) on \(B_R\), by the Sobolev inequality,
\[
\begin{split}
 \int_{B_R} H \abs{u}^2  \le  \int_{\R^N} H \abs{ \varphi u}^2  & \leq C \int_{\R^N} V \abs{(1-\varphi)u}^2 + \frac{1}{2}\int_{\R^N} \abs{\nabla (\varphi u) }^2 \\
 &\le C \int_{\R^N} V \abs{(1-\varphi)u}^2 + \int_{\R^N} \bigl( \abs{\varphi}^2\abs{\nabla u}^2 + \abs{\nabla \varphi}^2 \abs{u}^2\bigr).
\end{split}
\]
from which the conclusion follows.
\end{proof}

The proof of the counterpart of the Palais-Smale condition (Lemma~\ref{lem:PalaisSmale}) for \(N \in \{1, 2\}\) relies on the condition \(\lim_{\abs{x} \to \infty} V(x) \abs{x}^2 > 0\).

The rest of the proof is the same up to minor modifications when Lemma~\ref{lemmaPositivity} is used. These modification of the argument works in fact also for \(N \ge 3\) when one assumes that \(\lim_{\abs{x} \to \infty} V(x) \abs{x}^2 > 0\).


\begin{bibdiv}
\begin{biblist}

\bib{AmbrosettiBadialeCingolani1996}{article}{
   author={Ambrosetti, Antonio},
   author={Badiale, Marino},
   author={Cingolani, Silvia},
   title={Semiclassical states of nonlinear Schr\"odinger equations with
   bounded potentials},
   journal={Atti Accad. Naz. Lincei Cl. Sci. Fis. Mat. Natur. Rend. Lincei
   (9) Mat. Appl.},
   volume={7},
   date={1996},
   number={3},
   pages={155--160},
   issn={1120-6330},
}

\bib{AmbrosettiBadialeCingolani1997}{article}{
	author={Ambrosetti, Antonio},
	author = {Badiale, Marino},
	author = {Cingolani, Silvia},
	Journal = {Arch. Rational Mech. Anal.},
	Pages = {285\ndash 300},
	Title = {Semiclassical States of Nonlinear {S}chr\"{o}dinger Equations},
	Volume = {140},
	Year = {1997}}

\bib{AmbrosettiFelliMalchiodi2005}{article}{
    AUTHOR = {Ambrosetti, Antonio},
    AUTHOR = {Felli, Veronica},
    author = {Malchiodi, Andrea},
     TITLE = {Ground states of nonlinear {S}chr\"odinger equations with
              potentials vanishing at infinity},
   JOURNAL = {J. Eur. Math. Soc. (JEMS)},
    VOLUME = {7},
      YEAR = {2005},
    NUMBER = {1},
     PAGES = {117\ndash 144},
      ISSN = {1435-9855}
}

\bib{AmbrosettiMalchiodi2006}{book}{
      author={Ambrosetti, Antonio},
      author={Malchiodi, Andrea},
       title={{Perturbation methods and semilinear elliptic problems on $\R\sp
  n$}},
      series={Progress in Mathematics},
   publisher={Birkh\"{a}user Verlag},
        date={2006},
      volume={240},
}

\bib{AmbrosettiMalchiodiRuiz2006}{article}{
   author={Ambrosetti, A.},
   author={Malchiodi, A.},
   author={Ruiz, D.},
   title={Bound states of nonlinear Schr\"odinger equations with potentials
   vanishing at infinity},
   journal={J. Anal. Math.},
   volume={98},
   date={2006},
   pages={317--348},
   issn={0021-7670},
}

\bib{BaDengPeng2010}{article}{
   author={Ba, Na},
   author={Deng, Yinbin},
   author={Peng, Shuangjie},
   title={Multi-peak bound states for Schr\"odinger equations with compactly
   supported or unbounded potentials},
   journal={Ann. Inst. H. Poincar\'e Anal. Non Lin\'eaire},
   volume={27},
   date={2010},
   number={5},
   pages={1205--1226},
   issn={0294-1449},
}

\bib{BV}{article}{
   author={Bidaut-V{\'e}ron, Marie-Fran{\c{c}}oise},
   title={Local and global behavior of solutions of quasilinear equations of
   Emden-Fowler type},
   journal={Arch. Rational Mech. Anal.},
   volume={107},
   date={1989},
   number={4},
   pages={293--324},
   issn={0003-9527},
}

\bib{DCBonheureVanSchaftingen2008}{article}{
  author = {Bonheure, Denis},
  author = {Di Cosmo, Jonathan},
  author = {Van Schaftingen, Jean},
  title = {Nonlinear Schr\"odinger equation with unbounded or vanishing potentials: solutions concentrating on lower dimensional spheres}, 
  journal = {J. Differential Equations},
  volume = {252},
  year = {2012}, 
  number = {1}, 
  pages = {941--968},
}


\bib{BonheureVanSchaftingen2006}{article}{
   author={Bonheure, Denis},
   author={Van~Schaftingen, Jean},
   title={Nonlinear Schr\"odinger equations with potentials vanishing at
   infinity},
   journal={C. R. Math. Acad. Sci. Paris},
   volume={342},
   date={2006},
   number={12},
   pages={903--908},
   issn={1631-073X},
}

\bib{BonheureVanSchaftingen2008}{article}{
      author={Bonheure, Denis},
      author={Van~Schaftingen, Jean},
       title={Bound state solutions for a class of nonlinear {S}chr\"{o}dinger
  equations},
        date={2008},
     journal={Rev. Mat. Iberoamericana},
      volume={24},
       pages={297\ndash 351},
}



\bib{delPinoFelmer1996}{article}{
   author={del~Pino, Manuel},
   author={Felmer, Patricio},
   title={Local mountain passes for semilinear elliptic problems in
   unbounded domains},
   journal={Calc. Var. Partial Differential Equations},
   volume={4},
   date={1996},
   number={2},
   pages={121--137},
   issn={0944-2669},
}

\bib{delPinoFelmer1997}{article}{
	Author = {del~Pino, Manuel},
	Author = {Felmer, Patricio},
	Journal = {J. Funct. Anal.},
	Pages = {245\ndash 265},
	Title = {Semi-classical states for nonlinear {S}chr\"{o}dinger equations},
	Volume = {149},
	Year = {1997}}

\bib{delPinoFelmer1998}{article}{
    Author = {del~Pino, Manuel},
    Author = {Felmer, Patricio},
     TITLE = {Multi-peak bound states for nonlinear Schr\"odinger equations},
   JOURNAL = {Ann. Inst. H. Poincar\'e Anal. Non Lin\'eaire},
    VOLUME = {15},
      YEAR = {1998},
    NUMBER = {2},
     PAGES = {127\ndash 149},
}

\bib{delPinoFelmer2002}{article}{
	Author = {del~Pino, Manuel},
	Author = {Felmer, Patricio},
	Journal = {Math. Ann.},
	Pages = {1\ndash 32},
	Title = {Semi-classical states of nonlinear {S}chr\"{o}dinger equations: a variational reduction method},
	Volume = {324},
	Year = {2002}}

\bib{FY}{article}{
   author={Fei, Mingwen},
   author={Yin, Huicheng},
   title={Existence and concentration of bound states of nonlinear
   Schr\"odinger equations with compactly supported and competing
   potentials},
   journal={Pacific J. Math.},
   volume={244},
   date={2010},
   number={2},
   pages={261--296},
   issn={0030-8730},
}

\bib{FloerWeinstein1986}{article}{
   author={Floer, Andreas},
   author={Weinstein, Alan},
   title={Nonspreading wave packets for the cubic Schr\"odinger equation
   with a bounded potential},
   journal={J. Funct. Anal.},
   volume={69},
   date={1986},
   number={3},
   pages={397--408},
   issn={0022-1236},
}



\bib{Kwong}{article}{
	Author = {Kwong, Man-Kam},
	Journal = {Arch. Rational Mech. Anal.},
	Pages = {243\ndash 266},
	Title = {Uniqueness of positive solutions of {\(\Delta u - u + u^p = 0\) in \(\R^N\)}},
	Volume = {105},
	Year = {1989}}

\bib{Lions1984}{article}{
      author={Lions, P.-L.},
       title={The concentration-compactness principle in the calculus of
  variations. The locally compact case.},   
   partial={
     part={part 1},  
     journal={Ann. Inst. H. Poincar\'e Anal. Non Lin\'eaire},
     volume={1},
     date={1984},
     number={2},
     pages={109--145},
     issn={0294-1449},
   },
   partial={
        part={part 2},
        date={1984},
        ISSN={0294-1449},
     journal={Ann. Inst. H. Poincar\'e Anal. Non Lin\'eaire},
      volume={1},
      number={4},
       pages={223\ndash 283},},
      
}

\bib{MW}{book}{
    AUTHOR = {Mawhin, Jean},
    AUTHOR = {Willem, Michel},
     TITLE = {Critical point theory and {H}amiltonian systems},
    SERIES = {Applied Mathematical Sciences},
    VOLUME = {74},
 PUBLISHER = {Springer},
   ADDRESS = {New York},
      YEAR = {1989},
     PAGES = {xiv+277},
      ISBN = {0-387-96908-X}
}

\bib{MorozVanSchaftingen2009}{article}{
   author={Moroz, Vitaly},
   author={Van~Schaftingen, Jean},
   title={Existence and concentration for nonlinear Schr\"odinger equations
   with fast decaying potentials},
   journal={C. R. Math. Acad. Sci. Paris},
   volume={347},
   date={2009},
   number={15-16},
   pages={921--926},
   issn={1631-073X},
}


\bib{MorozVanSchaftingen2010}{article}{
      author = {Moroz, Vitaly},
      author = {Van~Schaftingen, Jean},
      title = {Semiclassical stationary states for nonlinear {S}chr\"odinger equations with fast decaying potentials},
      journal = {Calc. Var. Partial Differential Equations},
	Pages = {1\ndash 27},
	Volume = {37},
	Year = {2010},
	Number = {1},
}

\bib{Oh1988}{article}{
   author={Oh, Yong-Geun},
   title={Existence of semiclassical bound states of nonlinear Schr\"odinger
   equations with potentials of the class \((V)_a\)},
   journal={Comm. Partial Differential Equations},
   volume={13},
   date={1988},
   number={12},
   pages={1499--1519},
   issn={0360-5302},
}

\bib{Oh1988Errata}{article}{
   author={Oh, Yong-Geun},
   title={Correction to: ``Existence of semiclassical bound states of
   nonlinear Schr\"odinger equations with potentials of the class \((V)_a\)''},
   journal={Comm. Partial Differential Equations},
   volume={14},
   date={1989},
   number={6},
   pages={833--834},
   issn={0360-5302},
}


\bib{Oh1990}{article}{
   author={Oh, Yong-Geun},
   title={On positive multi-lump bound states of nonlinear Schr\"odinger
   equations under multiple well potential},
   journal={Comm. Math. Phys.},
   volume={131},
   date={1990},
   number={2},
   pages={223--253},
}

\bib{Rabinowitz1992}{article}{
   author={Rabinowitz, Paul H.},
   title={On a class of nonlinear Schr\"odinger equations},
   journal={Z. Angew. Math. Phys.},
   volume={43},
   date={1992},
   number={2},
   pages={270--291},
   issn={0044-2275},
}

\bib{Willem1996}{book}{
    AUTHOR = {Willem, Michel},
     TITLE = {Minimax theorems},
    SERIES = {Progress in Nonlinear Differential Equations and their Applications, 24},
 PUBLISHER = {Birkh\"auser},
   ADDRESS = {Boston, Mass.},
      YEAR = {1996},
     PAGES = {x+162},
      ISBN = {0-8176-3913-6}
}

\bib{Wang1993}{article}{
   author={Wang, Xuefeng},
   title={On concentration of positive bound states of nonlinear
   Schr\"odinger equations},
   journal={Comm. Math. Phys.},
   volume={153},
   date={1993},
   number={2},
   pages={229--244},
   issn={0010-3616},
}


\bib{YinZhang2009}{article}{
   author={Yin, Huicheng},
   author={Zhang, Pingzheng},
   title={Bound states of nonlinear Schr\"odinger equations with potentials
   tending to zero at infinity},
   journal={J. Differential Equations},
   volume={247},
   date={2009},
   number={2},
   pages={618--647},
   issn={0022-0396},
}



\end{biblist}
\end{bibdiv}
\end{document}